\documentclass[11pt,twoside,a4paper]{article}
\usepackage{color,amscd,amsmath,amssymb,amsthm,latexsym,stmaryrd,authblk,mathabx,shuffle,mathrsfs,hyperref,tikz,
tkz-tab} 
\usepackage[latin1]{inputenc}
\usepackage[T1]{fontenc}   
\usepackage[english]{babel}
\providecommand{\keywords}[1]{\textbf{\textit{Keywords.}} #1}
\providecommand{\AMSclass}[1]{\textbf{\textit{AMS classification.}} #1}

\input{xy}
\xyoption{all}

\title{Hopf-algebraic structures on mixed graphs}
\date{}
\author{Lo\"ic Foissy}
\affil{\small{Univ. Littoral Côte d'Opale, UR 2597
LMPA, Laboratoire de Mathématiques Pures et Appliquées Joseph Liouville
F-62100 Calais, France}.\\ Email: \texttt{foissy@univ-littoral.fr}}

\newcommand{\grun}{
\hspace{-1mm}\begin{array}{c}\begin{tikzpicture}[line cap=round,line join=round,>=angle 60,x=0.3cm,y=0.3cm]
\definecolor{cool}{rgb}{0.,0.8,0.}
\clip(3.4,-0.6) rectangle (4.6,0.6);
\draw [line width=2.pt,fill=black,fill opacity=1.0] (4.,0.) circle (0.03cm);
\end{tikzpicture}\end{array}\hspace{-1mm}}

\newcommand{\grdeuxvide}{
\hspace{-2mm}\begin{array}{c}\begin{tikzpicture}[line cap=round,line join=round,>=angle 60,x=0.3cm,y=0.3cm]
\clip(3.4,-0.6) rectangle (6.6,0.6);
\draw [line width=2.pt,fill=black,fill opacity=1.0] (4.,0.) circle (0.03cm);
\draw [line width=2.pt,fill=black,fill opacity=1.0] (6.,0.) circle (0.03cm);
\end{tikzpicture}\end{array}\hspace{-2mm}}

\newcommand{\grdeux}{
\hspace{-2mm}\begin{array}{c}\begin{tikzpicture}[line cap=round,line join=round,>=angle 60,x=0.3cm,y=0.3cm]
\clip(3.4,-0.6) rectangle (6.6,0.6);
\draw [line width=0.4pt] (4.2,0.)-- (5.8,0.);
\draw [line width=2.pt,fill=black,fill opacity=1.0] (4.,0.) circle (0.03cm);
\draw [line width=2.pt,fill=black,fill opacity=1.0] (6.,0.) circle (0.03cm);
\end{tikzpicture}\end{array}\hspace{-2mm}}

\newcommand{\grdeuxo}{
\hspace{-2mm}\begin{array}{c}\begin{tikzpicture}[line cap=round,line join=round,>=angle 60,x=0.3cm,y=0.3cm]
\clip(3.4,-0.6) rectangle (6.6,0.6);
\draw [->,line width=0.4pt] (4.2,0.)-- (5.8,0.);
\draw [line width=2.pt,fill=black,fill opacity=1.0] (4.,0.) circle (0.03cm);
\draw [line width=2.pt,fill=black,fill opacity=1.0] (6.,0.) circle (0.03cm);
\end{tikzpicture}\end{array}\hspace{-2mm}}

\newcommand{\grdeuxoo}{
\hspace{-2mm}\begin{array}{c}\begin{tikzpicture}[line cap=round,line join=round,>=angle 60,x=0.3cm,y=0.3cm]
\clip(3.4,-0.6) rectangle (6.6,0.6);
\draw [<->,line width=0.4pt] (4.2,0.)-- (5.8,0.);
\draw [line width=2.pt,fill=black,fill opacity=1.0] (4.,0.) circle (0.03cm);
\draw [line width=2.pt,fill=black,fill opacity=1.0] (6.,0.) circle (0.03cm);
\end{tikzpicture}\end{array}\hspace{-2mm}}

\newcommand{\grtrois}{
\hspace{-2mm}\begin{array}{c}\begin{tikzpicture}[line cap=round,line join=round,>=angle 60,x=0.3cm,y=0.3cm]
\clip(3.4,-0.6) rectangle (6.6,2.6);
\draw [line width=0.4pt] (4.2,0.)-- (5.8,0.);
\draw [line width=0.4pt] (4.1,0.15)-- (4.9,2.21);
\draw [line width=2.pt,fill=black,fill opacity=1.0] (4.,0.) circle (0.03cm);
\draw [line width=2.pt,fill=black,fill opacity=1.0] (6.,0.) circle (0.03cm);
\draw [line width=2.pt,fill=black,fill opacity=1.0] (5.,2.4) circle (0.03cm);
\end{tikzpicture}\end{array}\hspace{-2mm}}

\newcommand{\grtroisoe}{
\hspace{-2mm}\begin{array}{c}\begin{tikzpicture}[line cap=round,line join=round,>=angle 60,x=0.3cm,y=0.3cm]
\clip(3.4,-0.6) rectangle (6.6,2.6);
\draw [line width=0.4pt] (4.2,0.)-- (5.8,0.);
\draw [<-,line width=0.4pt] (4.1,0.15)-- (4.9,2.21);
\draw [line width=2.pt,fill=black,fill opacity=1.0] (4.,0.) circle (0.03cm);
\draw [line width=2.pt,fill=black,fill opacity=1.0] (6.,0.) circle (0.03cm);
\draw [line width=2.pt,fill=black,fill opacity=1.0] (5.,2.4) circle (0.03cm);
\end{tikzpicture}\end{array}\hspace{-2mm}}

\newcommand{\grtroiseo}{
\hspace{-2mm}\begin{array}{c}\begin{tikzpicture}[line cap=round,line join=round,>=angle 60,x=0.3cm,y=0.3cm]
\clip(3.4,-0.6) rectangle (6.6,2.6);
\draw [->,line width=0.4pt] (4.2,0.)-- (5.8,0.);
\draw [line width=0.4pt] (4.1,0.15)-- (4.9,2.21);
\draw [line width=2.pt,fill=black,fill opacity=1.0] (4.,0.) circle (0.03cm);
\draw [line width=2.pt,fill=black,fill opacity=1.0] (6.,0.) circle (0.03cm);
\draw [line width=2.pt,fill=black,fill opacity=1.0] (5.,2.4) circle (0.03cm);
\end{tikzpicture}\end{array}\hspace{-2mm}}

\newcommand{\grtroisooe}{
\hspace{-2mm}\begin{array}{c}\begin{tikzpicture}[line cap=round,line join=round,>=angle 60,x=0.3cm,y=0.3cm]
\clip(3.4,-0.6) rectangle (6.6,2.6);
\draw [line width=0.4pt] (4.2,0.)-- (5.8,0.);
\draw [<->,line width=0.4pt] (4.1,0.15)-- (4.9,2.21);
\draw [line width=2.pt,fill=black,fill opacity=1.0] (4.,0.) circle (0.03cm);
\draw [line width=2.pt,fill=black,fill opacity=1.0] (6.,0.) circle (0.03cm);
\draw [line width=2.pt,fill=black,fill opacity=1.0] (5.,2.4) circle (0.03cm);
\end{tikzpicture}\end{array}\hspace{-2mm}}

\newcommand{\grtroisoo}{
\hspace{-2mm}\begin{array}{c}\begin{tikzpicture}[line cap=round,line join=round,>=angle 60,x=0.3cm,y=0.3cm]
\clip(3.4,-0.6) rectangle (6.6,2.6);
\draw [->,line width=0.4pt] (4.2,0.)-- (5.8,0.);
\draw [<-,line width=0.4pt] (4.1,0.15)-- (4.9,2.21);
\draw [line width=2.pt,fill=black,fill opacity=1.0] (4.,0.) circle (0.03cm);
\draw [line width=2.pt,fill=black,fill opacity=1.0] (6.,0.) circle (0.03cm);
\draw [line width=2.pt,fill=black,fill opacity=1.0] (5.,2.4) circle (0.03cm);
\end{tikzpicture}\end{array}\hspace{-2mm}}

\newcommand{\grtroisoodeux}{
\hspace{-2mm}\begin{array}{c}\begin{tikzpicture}[line cap=round,line join=round,>=angle 60,x=0.3cm,y=0.3cm]
\clip(3.4,-0.6) rectangle (6.6,2.6);
\draw [->,line width=0.4pt] (4.2,0.)-- (5.8,0.);
\draw [->,line width=0.4pt] (4.1,0.15)-- (4.9,2.21);
\draw [line width=2.pt,fill=black,fill opacity=1.0] (4.,0.) circle (0.03cm);
\draw [line width=2.pt,fill=black,fill opacity=1.0] (6.,0.) circle (0.03cm);
\draw [line width=2.pt,fill=black,fill opacity=1.0] (5.,2.4) circle (0.03cm);
\end{tikzpicture}\end{array}\hspace{-2mm}}

\newcommand{\grtroisooo}{
\hspace{-2mm}\begin{array}{c}\begin{tikzpicture}[line cap=round,line join=round,>=angle 60,x=0.3cm,y=0.3cm]
\clip(3.4,-0.6) rectangle (6.6,2.6);
\draw [->,line width=0.4pt] (4.2,0.)-- (5.8,0.);
\draw [<->,line width=0.4pt] (4.1,0.15)-- (4.9,2.21);
\draw [line width=2.pt,fill=black,fill opacity=1.0] (4.,0.) circle (0.03cm);
\draw [line width=2.pt,fill=black,fill opacity=1.0] (6.,0.) circle (0.03cm);
\draw [line width=2.pt,fill=black,fill opacity=1.0] (5.,2.4) circle (0.03cm);
\end{tikzpicture}\end{array}\hspace{-2mm}}

\newcommand{\grtroisooodeux}{
\hspace{-2mm}\begin{array}{c}\begin{tikzpicture}[line cap=round,line join=round,>=angle 60,x=0.3cm,y=0.3cm]
\clip(3.4,-0.6) rectangle (6.6,2.6);
\draw [<-,line width=0.4pt] (4.2,0.)-- (5.8,0.);
\draw [<->,line width=0.4pt] (4.1,0.15)-- (4.9,2.21);
\draw [line width=2.pt,fill=black,fill opacity=1.0] (4.,0.) circle (0.03cm);
\draw [line width=2.pt,fill=black,fill opacity=1.0] (6.,0.) circle (0.03cm);
\draw [line width=2.pt,fill=black,fill opacity=1.0] (5.,2.4) circle (0.03cm);
\end{tikzpicture}\end{array}\hspace{-2mm}}

\newcommand{\grtroisoooo}{
\hspace{-2mm}\begin{array}{c}\begin{tikzpicture}[line cap=round,line join=round,>=angle 60,x=0.3cm,y=0.3cm]
\clip(3.4,-0.6) rectangle (6.6,2.6);
\draw [<->,line width=0.4pt] (4.2,0.)-- (5.8,0.);
\draw [<->,line width=0.4pt] (4.1,0.15)-- (4.9,2.21);
\draw [line width=2.pt,fill=black,fill opacity=1.0] (4.,0.) circle (0.03cm);
\draw [line width=2.pt,fill=black,fill opacity=1.0] (6.,0.) circle (0.03cm);
\draw [line width=2.pt,fill=black,fill opacity=1.0] (5.,2.4) circle (0.03cm);
\end{tikzpicture}\end{array}\hspace{-2mm}}

\newcommand{\grcompletun}{
\hspace{-2mm}\begin{array}{c}\begin{tikzpicture}[line cap=round,line join=round,>=angle 60,x=0.3cm,y=0.3cm]
\clip(3.4,-0.6) rectangle (6.6,2.6);
\draw [line width=0.4pt] (4.2,0.)-- (5.8,0.);
\draw [line width=0.4pt] (4.1,0.15)-- (4.9,2.21);
\draw [line width=0.4pt] (5.9,0.15)-- (5.1,2.21);
\draw [line width=2.pt,fill=black,fill opacity=1.0] (4.,0.) circle (0.03cm);
\draw [line width=2.pt,fill=black,fill opacity=1.0] (6.,0.) circle (0.03cm);
\draw [line width=2.pt,fill=black,fill opacity=1.0] (5.,2.4) circle (0.03cm);
\end{tikzpicture}\end{array}\hspace{-2mm}}

\newcommand{\grcompletdeux}{
\hspace{-2mm}\begin{array}{c}\begin{tikzpicture}[line cap=round,line join=round,>=angle 60,x=0.3cm,y=0.3cm]
\clip(3.4,-0.6) rectangle (6.6,2.6);
\draw [->,line width=0.4pt] (4.2,0.)-- (5.8,0.);
\draw [line width=0.4pt] (4.1,0.15)-- (4.9,2.21);
\draw [line width=0.4pt] (5.9,0.15)-- (5.1,2.21);
\draw [line width=2.pt,fill=black,fill opacity=1.0] (4.,0.) circle (0.03cm);
\draw [line width=2.pt,fill=black,fill opacity=1.0] (6.,0.) circle (0.03cm);
\draw [line width=2.pt,fill=black,fill opacity=1.0] (5.,2.4) circle (0.03cm);
\end{tikzpicture}\end{array}\hspace{-2mm}}

\newcommand{\grcomplettrois}{
\hspace{-2mm}\begin{array}{c}\begin{tikzpicture}[line cap=round,line join=round,>=angle 60,x=0.3cm,y=0.3cm]
\clip(3.4,-0.6) rectangle (6.6,2.6);
\draw [->,line width=0.4pt] (4.2,0.)-- (5.8,0.);
\draw [->,line width=0.4pt] (4.1,0.15)-- (4.9,2.21);
\draw [line width=0.4pt] (5.9,0.15)-- (5.1,2.21);
\draw [line width=2.pt,fill=black,fill opacity=1.0] (4.,0.) circle (0.03cm);
\draw [line width=2.pt,fill=black,fill opacity=1.0] (6.,0.) circle (0.03cm);
\draw [line width=2.pt,fill=black,fill opacity=1.0] (5.,2.4) circle (0.03cm);
\end{tikzpicture}\end{array}\hspace{-2mm}}

\newcommand{\grcompletquatre}{
\hspace{-2mm}\begin{array}{c}\begin{tikzpicture}[line cap=round,line join=round,>=angle 60,x=0.3cm,y=0.3cm]
\clip(3.4,-0.6) rectangle (6.6,2.6);
\draw [->,line width=0.4pt] (4.2,0.)-- (5.8,0.);
\draw [<-,line width=0.4pt] (4.1,0.15)-- (4.9,2.21);
\draw [line width=0.4pt] (5.9,0.15)-- (5.1,2.21);
\draw [line width=2.pt,fill=black,fill opacity=1.0] (4.,0.) circle (0.03cm);
\draw [line width=2.pt,fill=black,fill opacity=1.0] (6.,0.) circle (0.03cm);
\draw [line width=2.pt,fill=black,fill opacity=1.0] (5.,2.4) circle (0.03cm);
\end{tikzpicture}\end{array}\hspace{-2mm}}

\newcommand{\grcompletcinq}{
\hspace{-2mm}\begin{array}{c}\begin{tikzpicture}[line cap=round,line join=round,>=angle 60,x=0.3cm,y=0.3cm]
\clip(3.4,-0.6) rectangle (6.6,2.6);
\draw [->,line width=0.4pt] (4.2,0.)-- (5.8,0.);
\draw [<-,line width=0.4pt] (4.1,0.15)-- (4.9,2.21);
\draw [->,line width=0.4pt] (5.9,0.15)-- (5.1,2.21);
\draw [line width=2.pt,fill=black,fill opacity=1.0] (4.,0.) circle (0.03cm);
\draw [line width=2.pt,fill=black,fill opacity=1.0] (6.,0.) circle (0.03cm);
\draw [line width=2.pt,fill=black,fill opacity=1.0] (5.,2.4) circle (0.03cm);
\end{tikzpicture}\end{array}\hspace{-2mm}}

\newcommand{\grcompletsix}{
\hspace{-2mm}\begin{array}{c}\begin{tikzpicture}[line cap=round,line join=round,>=angle 60,x=0.3cm,y=0.3cm]
\clip(3.4,-0.6) rectangle (6.6,2.6);
\draw [->,line width=0.4pt] (4.2,0.)-- (5.8,0.);
\draw [<-,line width=0.4pt] (4.1,0.15)-- (4.9,2.21);
\draw [<-,line width=0.4pt] (5.9,0.15)-- (5.1,2.21);
\draw [line width=2.pt,fill=black,fill opacity=1.0] (4.,0.) circle (0.03cm);
\draw [line width=2.pt,fill=black,fill opacity=1.0] (6.,0.) circle (0.03cm);
\draw [line width=2.pt,fill=black,fill opacity=1.0] (5.,2.4) circle (0.03cm);
\end{tikzpicture}\end{array}\hspace{-2mm}}

\theoremstyle{plain}
\newtheorem{theo}{Theorem}[section]
\newtheorem{lemma}[theo]{Lemma}
\newtheorem{cor}[theo]{Corollary}
\newtheorem{prop}[theo]{Proposition}
\newtheorem{defi}[theo]{Definition}

\theoremstyle{remark}
\newtheorem{remark}{Remark}[section]
\newtheorem{notation}{Notations}[section]
\newtheorem{example}{Example}[section]

\newcommand{\rond}[1]{*++[o][F-]{#1}}
\newcommand{\K}{\mathbb{K}}
\newcommand{\N}{\mathbb{N}}
\newcommand{\calG}{\mathscr{G}}
\newcommand{\bfG}{\mathbf{G}}
\newcommand{\calO}{\mathcal{O}}

\newcommand{\arc}[1]{\stackrel{_{#1}}{\rightarrow}}
\newcommand{\arete}[1]{\stackrel{_{#1}}{\relbar}}
\newcommand{\id}{\mathrm{Id}}
\newcommand{\com}{\mathbf{Com}}
\newcommand{\bfP}{\mathbf{P}}
\newcommand{\bfQ}{\mathbf{Q}}

\newcommand{\eq}{\mathcal{E}}
\newcommand{\cl}{\mathrm{cl}}
\newcommand{\bfI}{\mathbf{I}}
\newcommand{\sym}{\mathfrak{S}}
\newcommand{\Char}{\mathrm{Char}}

\newcommand{\tdelta}{\tilde{\Delta}}

\newcommand{\vect}{\mathrm{Vect}}
\renewcommand{\ker}{\mathrm{Ker}}
\newcommand{\VPC}{\mathrm{VPC}}
\newcommand{\WVPC}{\mathrm{WVPC}}
\newcommand{\cc}{\mathrm{cc}}

\newcommand{\Endo}{\mathrm{End}}

\newcommand{\calF}{\mathcal{F}}
\newcommand{\caltopo}{\mathscr{T}\hspace{-1mm}\mathit{opo}}
\newcommand{\bftopo}{\mathbf{Topo}}
\newcommand{\calpos}{\caltopo_{T_0}}
\newcommand{\bfpos}{\bftopo_{T_0}}
\newcommand{\ehr}{\mathrm{ehr}}
\newcommand{\Ehr}{\mathrm{Ehr}}

\newcommand{\coinv}{\mathrm{coInv}}

\newcommand{\fix}{\mathrm{Fix}}

\begin{document}

\maketitle

\begin{abstract}
We introduce two coproducts on mixed graphs (that is to say graphs with both oriented and unoriented edges), 
the first one by separation of the vertices into two parts, and the second one given by contraction and extractions of subgraphs. 
We show that, with the disjoint union product, this gives a double bialgebra, that is to say that the first coproduct
makes it a Hopf algebra in the category of right comodules over the second coproduct. 

This structure implies the existence of a unique polynomial invariant on mixed graphs compatible with the product
and both coproducts: we prove that it is the (strong) chromatic polynomial of Beck, Bogart and Pham.
Using the action of the monoid of characters, we relate it to the weak chromatic polynomial, as well to Ehrhart polynomials
and to a polynomial invariants related to linear extensions. 
As applications, we give an algebraic proof of the link between the values of the strong chromatic polynomial
at negative values and acyclic orientations (a result due to Beck, Blado, Crawford, Jean-Louis and Young) 
and obtain a combinatorial description of the antipode of the Hopf algebra of mixed graphs. 
\end{abstract}

\\

\AMSclass{16T05 16T30 05C15 05C25}

\tableofcontents

\listoftables

\section*{Introduction}

Mixed graphs are graphs with both unoriented (which we will simply called edges) and oriented edges (which we will call arcs). They are used for example to study
scheduling problems with disjunctive (represented by unoriented edges) and precedence (represented by oriented edges) constraints
\cite{Sotskov2002,Sotskov2002,Hansen97}. In this context, a notion of admissible coloring, similar to the notion
used for classical graphs, gives a solution of the scheduling problem represented by the mixed graph.
These admissible colorings can be counted according to the number of colors, which gives a chromatic polynomial 
\cite{Beck2012,Beck2015-2}. 
The aim of this text is the study of this chromatic polynomial for mixed graphs in the context of double bialgebras,
as this has been done for graphs in \cite{Foissy36} and for posets and finite topologies in \cite{Foissy37}.
A double bialgebra is a family $(A,m,\Delta,\delta)$ such that:
\begin{itemize}
\item $(A,m,\delta)$ is a bialgebra.
\item $(A,m,\Delta)$ is a bialgebra in the category of right comodules over $(A,m,\delta)$, with the coaction given by $\delta$
itself.
\end{itemize}
In particular, $\delta$ and $\Delta$ satisfies the following compatibility:
\[(\Delta \otimes \id)\circ \delta=(\id \otimes \id \otimes m)\circ (\id \otimes c\otimes \id)\circ (\delta \otimes \delta) \circ \Delta,\]
where $c:A\otimes A\longrightarrow A\otimes A$ is the usual flip. The counit $\varepsilon_\Delta$ of $\Delta$ satisfies the following compatibility with $\delta$: for any $a\in A$,
\[(\varepsilon_\Delta \otimes \id)\circ \delta(x)=\varepsilon_\Delta(x)1_A.\] 
A simple example of such an object is $\K[X]$, with its usual algebra structures and its multiplicative coproducts defined by
\begin{align*}
\Delta(X)&=X\otimes 1+1\otimes X,&\delta(X)&=X\otimes X.
\end{align*}
Double bialgebras play an important role in the study of rough paths and regularity structures, used in the study of stochastics PDEs, with examples based on families of rooted trees with decorations \cite{BrunedChevyrev2019,Bruned2019}.
Other examples of double bialgebras can be found on various families of graphs \cite{
Manchon2012}, posets or finite topologies \cite{Foissy27,Foissy31,Foissy34}, graphs \cite{Manchon2012,Calaque2011,Foissy36}, hypergraphs \cite{Ebrahimi-Fard2024,Foissy44}$\ldots$
We proved in \cite{Foissy37,Foissy36} the following results: if $(A,m,\Delta,\delta)$ is a double bialgebra, under a condition
of connectedness of $(A,m,\Delta)$,
\begin{itemize}
\item There exists a unique double bialgebra morphism $P_0:(A,m,\Delta,\delta)\longrightarrow (\K[X],m,\Delta,\delta)$,
which can be explicitly described with iterations of the coproduct $\Delta$ of $A$ and with the counit $\epsilon_\delta$
of its coproduct $\delta$.
\item Let us denote by $\Char(A)$ the set of characters of $A$. It inherits a convolution product $\star$, dual to $\delta$.
We denote by $\Endo_B(A,\K[X])$ the set of bialgebra morphisms from $(A,m,\Delta)$ to $(\K[X],m,\Delta)$. 
Then \cite[Proposition 2.5]{Foissy40}, the following defines an action of the monoid $(\Char(A),\star)$ on $\Endo_B(A,\K[X])$:
\begin{align*}
\leftsquigarrow&:\left\{\begin{array}{rcl}
\Endo_B(A,\K[X])\times \Char(A)&\longrightarrow&\Endo_B(A,\K[X])\\
(\phi,\lambda)&\longmapsto& \phi\leftsquigarrow \lambda=(\phi\otimes \lambda)\circ \delta.
\end{array}\right.
\end{align*}
Moreover, the two following maps are bijections, inverse to each other:
\begin{align}\label{eqtheta}
\theta&:\left\{\begin{array}{rcl}
\Char(A)&\longrightarrow&\Endo_B(A,\K[X])\\
\lambda&\longmapsto&P_0\leftsquigarrow \lambda,
\end{array}\right.&
\theta^{-1}&:\left\{\begin{array}{rcl}
\Endo_B(A,\K[X])&\longrightarrow&\Char(A)\\
\phi&\longmapsto&\epsilon_\delta \circ \phi,
\end{array}\right.
\end{align}
where here $\epsilon_\delta$ is the counit of $(\K[X],m,\delta)$, which sends any $P\in \K[X]$ onto $P(1)$. 
\item When $(A,m,\Delta)$ is a graded and connected bialgebra, then under a technical condition on $\delta$,
the set of homogeneous bialgebra morphisms from $(A,m,\Delta)$ to $(\K[X],m,\Delta)$ is in bijection
with the dual of the homogeneous component $A_1$ of $A$ of degree $1$. 
\item If $(A,m,\Delta)$ is a Hopf algebra, then its antipode is given by 
\[S=(\epsilon_\delta^{*-1}\otimes \id)\circ \delta,\]
where $\epsilon_\delta^{*-1}$ is the inverse of the counit $\epsilon_\delta$ of $\delta$ for the convolution product
$*$ associated to $\Delta$. Moreover, for any $a\in A$,
\[\epsilon_\delta^{*-1}(a)=P_0(a)(-1).\]
\end{itemize}

We apply here these results on mixed graphs. In the second section, we define a structure of double bialgebras on mixed graphs. 
We work in the species framework and use the formalism built in \cite{Foissy41} of contraction-extraction coproduct, which is shortly described in the first section.
We first give in Proposition \ref{prop1.4} to the species of mixed graphs $\bfG$ a bialgebra structure in the category of species 
(what is commonly called a twisted bialgebra structure), and then a contraction-extraction coproduct in Proposition \ref{prop1.6}.
Consequently, applying the bosonic Fock functor \cite{Aguiar2010}, we obtain a double bialgebra of mixed graphs
$\calF[\bfG]$, and more generally, for any commutative and cocommutative bialgebra $V$, a double bialgebra $\calF_V[\bfG]$
of mixed graphs whose vertices are decorated by elements of $V$. Using Loday and Ronco's rigidity theorem, 
we prove that $(\calF_V[\bfG],\Delta)$ is a cofree coalgebra (Corollary \ref{cor1.9}). 

We study certain sub-objects and quotients of $\bfG$ in the third section. Obviously, simple (i.e., unoriented) graphs
and oriented graphs define twisted double subbialgebras of $\bfG$, denoted respectively by $\bfG_s$ and $\bfG_o$
(Proposition 2.1). Applying the bosonic Fock functor to $\bfG_s$, we obtain again the double bialgebra of graphs  
of \cite{Foissy36,Foissy41}. We also consider the subspecies of  acyclic mixed graphs $\bfG_{aco}$, 
which turns out to be stable under the product and the first coproduct $\Delta$, but not on the second one.
However, quotienting by non acyclic mixed graphs, $\bfG_{ac}$ can be seen as a twisted double bialgebra,
quotient of $\bfG$ (Proposition 2.5). It contains a twisted double subbialgebra of oriented acyclic mixed graphs
$\bfG_{aco}$, which has itself for quotient the twisted double bialgebra of topologies of \cite{Foissy37} (Proposition 2.6). 

The fourth part is devoted to polynomial invariants of mixed graphs, that is to say bialgebra morphisms
from the bialgebra of mixed graphs $(\calF[\bfG],m,\Delta)$ to $(\K[X],m,\Delta)$. We first describe
the unique polynomial invariant compatible with the second coalgebraic structure of mixed graphs:
it turns out to be the strong chromatic polynomial $P_{chr_S}$ of \cite{Beck2012},
see Proposition \ref{prop3.2}. In other words, for any mixed graph $G$, for any $n\geq 1$, $P_{chr_S}(G)(n)$
is the number of $n$-valid colorings of $G$, that is to say maps $c$ from the set of vertices $V(G)$ of $G$ to $\{1,\ldots,n\}$
such that, for any pair of vertices $x,y$ of $G$,
\begin{itemize}
\item If $x$ and $y$ are related by an edge of $G$, then $c(x)\neq c(y)$.
\item If $x$ and $y$ are related by an arc of $G$, then $c(x)<c(y)$. 
\end{itemize}
A notion of weak valid coloring is also defined in \cite{Beck2012}. A weak $n$-valid coloring is 
a map $c$ from the set of vertices $V(G)$ of $G$ to $\{1,\ldots,n\}$ such that, for any pair of vertices $x,y$ of $G$,
\begin{itemize}
\item If $x$ and $y$ are related by an edge of $G$, then $c(x)\neq c(y)$.
\item If $x$ and $y$ are related by an arc of $G$, then $c(c)\leq c(y)$. 
\end{itemize}
The polynomial counting the number of weak $n$-valid colorings of $G$ is denoted by $P_{chr_W}(G)$. 
Using the the action of the monoid of characters described earlier and the character of $\calF[\bfG]$ defined by
\[\lambda_W(G)=\begin{cases}
1\mbox{ if $G$ is an oriented graph},\\
0\mbox{ otherwise},
\end{cases}\]
we obtain that $P_{chr_W}=P_{chr_S}\leftsquigarrow \lambda_W$, which implies that $P_{chr_w}$
is a bialgebra morphism from $(\calF[\bfG],m,\Delta)$ to $(\K[X],m,\Delta)$, see Corollary \ref{cor3.4}.
Finally, using the correspondence between homogeneous polynomial invariants and elements of $\calF[\bfG]_1$,
we construct a homogeneous bialgebra morphism $P_0$ from $(\calF[\bfG],m,\Delta)$ to $(\K[X],m,\Delta)$,
related to the number of linear extensions (Corollary \ref{cor3.5}) and to a character $\lambda_0$.
After the determination of invertible characters of $\calF[\bfG]$ for the convolution product $\star$ dual to $\delta$,
we prove that both characters $\lambda_0$ and $\lambda_W$ are invertible, which allows to express
$P_{chr_S}$ in terms of $P_{chr_W}$ or $P_0$, with the help of certain characters $\nu_W$ and $\mu_W$
(Proposition \ref{prop3.13}). This allows to give a formula for the leading monomial of $P_{chr_W}(G)$ and $P_{chr_S}(G)$
in Corollary \ref{cor3.14}, with coefficients (in the case of $P_{chr_S}(G)$) related to Murua's coefficients \cite{Murua2006}. \\

In the fifth section, we give an algebraic proof of the result \cite{Beck2012},
which gives a combinatorial interpretation of $P_{chr_S}(G)(-1)$ in terms of acyclic orientations. 
We firstly introduce a surjective double bialgebra morphism $\Theta$ from $\calF[\bfG]$ to the double bialgebra
of acyclic oriented mixed graphs $\calF[\bfG_{aco}]$ in Theorem \ref{theo4.2}. The unicity of the double bialgebra morphism
to $\K[X]$ immediately implies for example that the strong chromatic polynomial of a graph $G$ is the sum of 
the strong chromatic polynomial of all its acyclic orientations (Corollary \ref{cor4.3}). 
Introducing two one-parameter families of characters of $\calF[\bfG]$ in Proposition \ref{prop4.4},
we introduce two new polynomial invariants, which turn out to be Ehrhart polynomials and satisfy a duality principle
on acyclic mixed graphs (Corollary \ref{cor4.7}). Mixing this duality principle with the morphism $\Theta$,
we obtain a new proof that for any mixed graph, $P_{chr_S}(G)(-1)$ is the number of acyclic orientations of $G$,
up to a sign (Corollary \ref{cor4.9}). This allows to give a formula for the antipode of $(\calF[\bfG],m,\Delta)$
involving the number of acyclic orientations of $G$, see Corollary \ref{cor4.10}. \\

The last section is devoted to combinatorial interpretations of special characters of mixed graphs.
We give an algebraic proof of  a result of \cite{Beck2012} in an algebraic way (Proposition \ref{prop5.1} and Corollary \ref{cor5.2}) 
about values of the weak chromatic polynomial on negative integers (for totally mixed graphs only),
and we give a combinatorial interpretation of $\nu_W(G)$ when $G$ is a simple graph or an oriented graph 
(Proposition \ref{prop5.3}), where $\nu_W$ is the inverse of $\lambda_W$ for the convolution product $\star$. \\

The paper ends with an appendix giving formulas for the number of $N$-decorated mixed, oriented or simple graphs with a fixed number of vertices. These numbers are the the dimensions of the homogeneous components of $\calF_V[\bfG]$, when $V$ is finite-dimensional. \\

\textbf{Thanks}. The author thanks the two anonymous referees for their useful comments, which help to improve the redaction of this text.\\

\textbf{Acknowledgments}. 
The author acknowledges support from the grant ANR-20-CE40-0007
\emph{Combinatoire Algébrique, Résurgence, Probabilités Libres et Opérades}.\\

\begin{notation} \begin{enumerate}
\item We denote by $\K$ a commutative field of characteristic zero. Any vector space in this field will be taken over $\K$.
\item For any $n\in \N$, we denote by $[n]$ the set $\{1,\ldots,n\}$. In particular, $[0]=\emptyset$.
\item If $(C,\Delta)$ is a (coassociative but not necessarily counitary) coalgebra, 
we denote by $\Delta^{(n)}$ the $n$-th iterated coproduct of $C$:
$\Delta^{(1)}=\Delta$ and if $n\geq 2$,
\[\Delta^{(n)}=\left(\Delta \otimes \id^{\otimes (n-1)}\right)\circ \Delta^{(n-1)}:C\longrightarrow C^{\otimes (n+1)}.\]
\item If $(B,m,\Delta)$ is a bialgebra of unit $1_B$ and of counit $\varepsilon_B$, let us denote by $B_+=\ker(\varepsilon_B)$ 
its augmentation ideal. We define a coproduct on $B_+$ by
\begin{align*}
&\forall x\in B_+,&\tdelta(x)=\Delta(x)-x\otimes 1_B-1_B\otimes x.
\end{align*}
Then $(B_+,\tdelta)$ is a coassociative (not necessarily counitary) coalgebra. 
\end{enumerate}\end{notation}

\section{Species, twisted bialgebras and contractions}

\subsection{Species and twisted bialgebras}

Recall that a species \cite{Joyal1981,Joyal1986} is a functor from the category of finite sets with bijections, to the category of vector spaces. We shall use the following notations: if $\bfP$ is a species,
\begin{itemize}
\item For any finite set $X$, the vector space associated to $X$ by $\bfP$ is denoted by $\bfP[X]$.
We will shorty write $\bfP[n]$ instead of $\bfP\left[[n]\right]$ for any $n\in \N$.
\item For any bijection $\sigma:X\longrightarrow Y$  between two finite sets, the linear map associated to $\sigma$ by $\bfP$
is denoted by $\bfP[\sigma]:\bfP[X]\longrightarrow \bfP[Y]$. 
\end{itemize}
With these notations, if $\sigma:X\longrightarrow Y$ and $\tau:Y\longrightarrow Z$ are two bijections between finite sets, then $\bfP[\tau\circ \sigma]=\bfP[\tau]\circ \bfP[\sigma]$,
and for any finite set $X$, $\bfP[\id_X]=\id_{\bfP[X]}$.\\

If $\bfP$ and $\bfQ$ are two species, a morphism $f:\bfP\longrightarrow \bfQ$ is a natural transformation from $\bfP$ to $\bfQ$. In other terms, for any finite $X$,
$f_X:\bfP[X]\longrightarrow \bfQ[X]$ is a linear map such that for any bijection $\sigma:X\longrightarrow Y$ between two finite sets, the following diagram commutes:
\[\xymatrix{\bfP[X]\ar[r]^{\bfP[\sigma]} \ar[d]_{f_X}&\bfP[Y]\ar[d]^{f_Y}\\
\bfQ[X]\ar[r]_{\bfQ[\sigma]}&\bfQ[X]}\]

The Cauchy tensor product of species is denoted by $\otimes$: if $\bfP$ and $\bfQ$ are two species, for any finite set $X$,
\[\bfP\otimes \bfQ[X]=\bigoplus_{X=Y\sqcup Z} \bfP[Y]\otimes \bfQ[Z].\]
If $\sigma:X\longrightarrow X'$  is a bijection between two finite sets, then 
\[\bfP\otimes \bfQ[\sigma]=\bigoplus_{X=Y\sqcup Z} \bfP[\sigma_{\mid Y}]\otimes \bfQ[\sigma_{\mid Z}].\]
With this tensor product, the category of species is symmetric monoidal, with the flip defined from $\bfP\otimes \bfQ$ to $\bfQ\otimes \bfP$ by
\[c_{\bfP[X],\bfP[Y]}=
(c_{\bfP\otimes \bfQ})_{\mid \bfP[X]\otimes \bfQ[Y]}=
\left\{\begin{array}{rcl}
\bfP[X]\otimes \bfQ[Y]&\longrightarrow&\bfQ[Y]\otimes \bfP[X]\\
x\otimes y&\longmapsto&y\otimes x. 
\end{array}\right.\]

A twisted algebra (resp. coalgebra, bialgebra) is an algebra (resp. coalgebra, bialgebra) in this symmetric monoidal category. Let us now precise what this means. A twisted algebra $\bfP$ comes with maps $m_{X,Y}:\bfP[X]\otimes \bfP[Y]\longrightarrow \bfP[X\sqcup Y]$ 
 for any finite sets $X$ and $Y$, such that:
\begin{itemize}
\item ($m$ is a species morphism). For any bijections $\sigma:X\longrightarrow X'$ and $\tau:Y\longrightarrow Y'$, where $(X,X')$ and $(Y,Y')$ are pairs of disjoint sets, the
following diagram commutes:
\[\xymatrix{\bfP[X]\otimes \bfP[Y]\ar[rr]^{m_{X,Y}}\ar[d]_{\bfP[\sigma]\otimes \bfP[\tau]}
&&\bfP[X\sqcup Y]\ar[d]^{^\bfP[\sigma\sqcup \tau]}\\
\bfP[X']\otimes \bfP[Y']\ar[rr]_{m_{X',Y'}}&&\bfP[X'\sqcup Y']}\]
where
\[\sigma\sqcup \tau:\left\{\begin{array}{rcl}
X\sqcup Y&\longrightarrow&X'\sqcup Y'\\
x\in X&\longrightarrow&\sigma(x),\\
y\in Y&\longrightarrow&\tau(y).
\end{array}\right.\]
\item (Associativity of $m$). For any disjoint finite sets $X,Y,Z$, the following diagram commutes:
\[\xymatrix{\bfP[X]\otimes \bfP[Y]\otimes \bfP[Z]\ar[d]_{\id_{\bfP[X]}\otimes m_{Y,Z}}
\ar[rr]^{m_{X,Y}\otimes \id_{\bfP[Z]}}&&\bfP[X\sqcup Y]\otimes \bfP[Z]\ar[d]^{m_{X\sqcup Y,Z}}\\
\bfP[X]\otimes \bfP[Y\sqcup Z]\ar[rr]_{m_{X,Y\sqcup Z}}&&\bfP[X\sqcup Y\sqcup Z]}\]
\item There exists an element $1_\bfP\in \bfP[\emptyset]$ such that for any finite set $X$, for any $x\in \bfP[X]$,
\[m_{\emptyset,X}(1_\bfP\otimes x)=m_{X,\emptyset}(x\otimes 1_\bfP)=x.\]
\item (Commutativity of $m$). We shall say that $\bfP$ is commutative if, for any finite disjoint sets $X,Y$, in $\bfP[X\sqcup Y]$,
\begin{align*}
&\forall x\in \bfP[X], \:\forall y\in \bfP[Y],&m_{X,Y}(x\otimes y)=m_{Y,X}(y\otimes x).
\end{align*}
\end{itemize}

A twisted coalgebra $\bfP$ comes with maps  $\Delta_{X,Y}:\bfP[X\sqcup Y]\longrightarrow \bfP[X]\otimes \bfP[Y]$
for any disjoint finite sets $X$ and $Y$, such that:
\begin{itemize}
\item ($\Delta$ is a species morphism). For any bijections $\sigma:X\longrightarrow X'$ and $\tau:Y\longrightarrow Y'$, where $(X,Y)$ and $(X',Y')$ are pairs of disjoint finite sets, 
the following diagram commutes:
\[\xymatrix{\bfP[X\sqcup Y]\ar[rr]^{\Delta_{X,Y}}\ar[d]_{\bfP[\sigma\sqcup \tau]}
&&\bfP[X]\otimes \bfP[Y]\ar[d]^{\bfP[\sigma]\otimes \bfP[\tau]}\\
\bfP[X'\sqcup Y']\ar[rr]_{\Delta_{X',Y'}}&&\bfP[X']\otimes \bfP[Y']}\]
\item (Coassociativity of $\Delta$). For any disjoint finite sets $X,Y,Z$, the following diagram commutes:
\[\xymatrix{\bfP[X\sqcup Y\sqcup Z]\ar[rr]^{\Delta_{X\sqcup Y,Z}} \ar[d]_{\Delta_{X\sqcup Y,Z}}
&&\bfP[X\sqcup Y]\otimes \bfP[Z]\ar[d]^{\Delta_{X,Y}\otimes \id_{\bfP[Z]}}\\
\bfP[X]\otimes \bfP[Y\sqcup Z]\ar[rr]_{\id_{\bfP[X]}\otimes \Delta_{Y,Z}}&&\bfP[X]\otimes \bfP[Y]\otimes \bfP[Z]}\]
\item (Counit). There exists a linear map $\varepsilon_\Delta:\bfP[\emptyset]\longrightarrow \K$ 
such that for any finite set $X$, the following diagram commutes:
\[\xymatrix{\bfP[\emptyset]\otimes \bfP[X]\ar[rrd]_{\varepsilon_\Delta\otimes \id_{\bfP[X]}}
&&\bfP[X] \ar[ll]_{\Delta_{\emptyset,X}}  \ar[rr]^{\Delta_{X,\emptyset}}
\ar[d]_{\id_{\bfP[X]}}&&\bfP[X]\otimes \bfP[\emptyset]\ar[lld]^{\hspace{3mm}\id_{\bfP[X]}\otimes \varepsilon_\Delta}\\
&&\bfP[X]&&}\]
\item We shall say that $\bfP$ is cocommutative if for any disjoint finite sets $X,Y$,
\[\Delta_{Y,X}=c_{\bfP[X],\bfP[Y]}\circ \Delta_{X,Y}.\]
\end{itemize}
A twisted bialgebra is both a twisted coalgebra and a twisted bialgebra, such that the counit and the coproduct are algebra morphisms. In other words, if $X$ is a finite set and $X=I\sqcup J=I'\sqcup J'$, the following diagram commutes: 
\[\xymatrix{\bfP[I']\otimes \bfP[J'] \ar[rrr]^{m_{I',J'}}\ar[d]_{\Delta_{I'\cap I,I'\cap J}\otimes
\Delta_{J'\cap I,J'\cap J}}&&&\bfP[I'\sqcup J']=\bfP[I\sqcup J]\ar[dd]^{\Delta_{I,J}}\\
\bfP[I'\cap I]\otimes \bfP[I'\cap J]\otimes \bfP[J'\cap I]\otimes \bfP[J'\cap J]
\ar[d]_{\id_{\bfP[I'\cap I]}\otimes c_{\bfP[I'\cap J],\bfP[J'\cap I]}\otimes \id_{\bfP[J'\cap J]}}&&&\\
\bfP[I'\cap I]\otimes \bfP[J'\cap I]\otimes \bfP[I'\cap J]\otimes \bfP[J'\cap J]
\ar[rrr]_{\hspace{3cm}m_{I'\cap I,J'\cap I}\otimes m_{I'\cap J,J'\cap J}}&&&\bfP[I]\otimes \bfP[J]
}\]
or equivalently
\begin{align*}
\Delta_{I,J}\circ m_{I,'J'}&=(m_{I'\cap I,J'\cap I}\otimes m_{I'\cap J,J'\cap J})\circ (\id_{\bfP[I\cap I]}
\otimes c_{\bfP[I'\cap J],\bfP[J'\cap I]}\otimes \id_{\bfP[J'\cap J]})\\
&\circ (\Delta_{I'\cap I,I'\cap J}\otimes \Delta_{J'\cap I,J'\cap J}).
\end{align*}
The compatibility between the coproduct and the unit is written as $\Delta_\bfP(1_\bfP)=1_\bfP\otimes 1_\bfP$
and the compatibility between the counit and the product is equivalent to
\begin{align*}
&\forall x,y\in \bfP[\emptyset],&\varepsilon_\Delta(xy)=\varepsilon_\Delta(x)\varepsilon_\Delta(y).
\end{align*}

Let $V$ be a vector space. The $V$-colored Fock functor $\calF_V$, defined in \cite[Definition 3.2]{Foissy41},
 sends any species $\bfP$ to
\begin{align*}
\calF_V[\bfP]&=\bigoplus_{n=0}^\infty \coinv(V^{\otimes n}\otimes \bfP[n])\\
&=\bigoplus_{n=0}^\infty \frac{V^{\otimes n}\otimes \bfP[n]}
{\vect(v_1\ldots v_n \otimes \bfP[\sigma](p)-v_{\sigma(1)}\ldots v_{\sigma(n)}\otimes p\mid
\sigma\in \sym_n,\: p\in \bfP[n],\:v_1,\ldots,v_n\in V)}\\
&=\bigoplus_{n=0}^\infty V^{\otimes n}\otimes_{\sym_n} \bfP[n].
\end{align*}
If $f:\bfP\longrightarrow \bfQ$ is a species morphism, then
\[\calF_V(f):\left\{\begin{array}{rcl}
\calF_V[\bfP]&\longmapsto&\calF_V[\bfQ]\\
\overline{v_1\ldots v_n \otimes p}&\longmapsto&\overline{v_1\ldots v_n \otimes f(p)}.
\end{array}\right.\]
Note that $\calF_V$ is a functor of symmetric monoidal categories, that is to say it sends the Cauchy tensor product of two species to the tensor products of the associated vector spaces,
and that it is compatible with the flip.  Consequently, it sends a twisted algebra, coalgebra or bialgebra to an algebra, coalgebra or bialgebra.

\begin{remark}
When $V=\K$, we obtain the bosonic Fock functor of \cite{Aguiar2010}:
\[\calF[\bfP]=\bigoplus_{n=0}^\infty \coinv(\bfP[n])
=\bigoplus_{n=0}^\infty \frac{\bfP[n]}{\vect(\bfP[\sigma](p)-p\mid \sigma\in \sym_n,\: p\in \bfP[n])}.\]
\end{remark}

\subsection{Contraction-extraction coproducts}

Let us now recall the contraction-extraction coproduct of \cite{Foissy41}. We shall need the following notations:

\begin{notation} 
The species $\com$ is defined by $\com[X]=\K$ for any finite set $X$ and $\com[\sigma]=\id_\K$
for any bijection $\sigma$ between two finite sets. 
 \end{notation}

\begin{notation}
\begin{enumerate}
\item Let $X$ be a finite set. We denote by $\eq[X]$ the set of equivalences on $X$. 
\item For any bijection $\sigma:X\longrightarrow Y$ between two finite sets, for any $\sim\in \eq[X]$,
we define an equivalence  $\sim_\sigma$ on $Y$ by
\begin{align*}
&\forall y,y'\in Y,&y\sim_\sigma y'&\Longleftrightarrow \sigma^{-1}(y)\sim \sigma^{-1}(y').
\end{align*}
This defines a bijection from $\eq[X]$ to $\eq[Y]$. 
Moreover, $\sigma$ induces a bijection between $X/\sim$ and $Y/\sim_\sigma$, sending $C\in X/\sim$
to $\sigma(C)\in Y/\sim_\sigma$.
\item The set $\eq[X]$ is partially ordered by the refinement order: 
\begin{align*}
&\forall \sim,\sim'\in \eq[X],&\sim\leq \sim'&\Longleftrightarrow (\forall x,y\in X,\: x\sim' y\Longrightarrow x\sim y),
\end{align*}
that is to say, $\sim\leq \sim'$ if the classes of $\sim$ are unions of classes of $\sim'$.
\item If $\sim'\in \eq[X]$, there is a natural bijection from $\eq[X/\sim']$ to $\{\sim\in \eq[X],\:\sim\leq \sim'\}$.
From now, we identify $\eq[X/\sim']$ and  $\{\sim\in \eq[X],\:\sim\leq \sim'\}$, as well as
$(X/\sim')/\overline{\sim}$ and $X/\sim$ through these bijections.
\end{enumerate} \end{notation}

Let $(\bfP,m,\Delta)$ be a twisted bialgebra. A contraction-extraction coproduct on $\bfP$ is a family $\delta$ of maps such that 
for any finite set $X$, for any $\sim\in \eq[X]$, $\delta_\sim:\bfP[X]\longrightarrow \bfP[X/\sim]\otimes \bfP[X]$, with the following conditions:
\begin{itemize}
\item (Compatibility with the species structure \cite[Definition 2.2, second item]{Foissy41}).  For any bijection $\sigma:X\longrightarrow Y$ between two finite sets, for any $\sim\in \eq[X]$,
the following diagram commutes:
\[\xymatrix{\bfP[X]\ar[r]^<(.2){\delta_\sim}\ar[d]_{\bfP[\sigma]}&\bfP[X/\sim]\otimes \bfP[X]
\ar[d]^{\bfP[\sigma/\sim]\otimes \bfP[\sigma]}\\
\bfP[Y]\ar[r]_<(.2){\delta_{\sim_\sigma}}&\bfP[Y/\sim_\sigma]\otimes \bfP[Y]}\]
\item (Coassociativity of $\delta$ \cite[Definition 2.2, third item]{Foissy41}). If $X$ is a finite set and $\sim\leqslant \sim'\in \eq[X]$, the following diagram commutes:
\[\xymatrix{\bfP[X]\ar[r]^<(.3){\delta_{\sim'}} \ar[d]_{\delta_\sim}&\bfP[X/\sim']\otimes \bfP[X]\ar[d]^{\delta_\sim\otimes \id}\\
\bfP[X/\sim]\otimes \bfP[X]\ar[r]_<(0.15){\id \otimes \delta_{\sim'}}&\bfP[X/\sim]\otimes \bfP[X/\sim']\otimes \bfP[X]}\]
Moreover, if $\sim,\sim'\in \eq[X]$, such that we do not have $\sim\leq \sim'$, then $(\id \otimes \delta_{\sim'})\circ \delta_\sim=0$.
\item (Counit \cite[Definition 2.2, fourth item]{Foissy41}). There exists a twisted algebra morphism $\epsilon_\delta:\bfP\longrightarrow \com$ such that for any finite set $X$,
for any $\sim\in \eq[X]$,
\[(\id \otimes {\epsilon_\delta}_X)\circ \delta_\sim=\begin{cases}
\id_{\bfP[X]}\mbox{ if $\sim$ is the equality of $X$},\\
0\mbox{ otherwise},
\end{cases}\]
and
\[\sum_{\sim \in \eq[X]} ({\epsilon_\delta}_{X/\sim}\otimes \id)\circ \delta_\sim=\id_{\bfP[X]}.\]
\item  (Compatibility with the algebraic structure \cite[Proposition 2.4]{Foissy41}). For any finite set $X$, for any $\sim\in \eq[X]$, putting $\sim_X=\sim\cap X^2$ and $\sim_Y=\sim\cap Y^2$,
\begin{align*}
\delta_\sim\circ m_{X,Y}&=\begin{cases}
(m_{X/\sim_X,Y/\sim_Y}\otimes m_{X,Y})\circ (\id \otimes c \otimes \id)\circ (\delta_{\sim_X}\otimes \delta_{\sim_Y})
\mbox{ if }\sim=\sim_X\sqcup \sim_Y,\\
0\mbox{ otherwise},
\end{cases}
\end{align*}
Moreover, $\delta_{\sim_\emptyset}(1_\bfP)=1_\bfP\otimes 1_\bfP$, where $\sim_\emptyset$ is the unique equivalence on $\emptyset$.
\item  (Compatibility with the coalgebraic structure \cite[Proposition 2.4]{Foissy41}). For any finite set $X\sqcup Y$, for any $\sim\in \eq[X\sqcup Y]$, putting $\sim_X=\sim\cap X^2$ 
and $\sim_Y=\sim\cap Y^2$,
\begin{align*}
(\Delta_{X/\sim_X,Y/\sim_Y}\otimes \id)\circ \delta_\sim&=
m_{1,3,24}\circ (\delta_{\sim_X}\otimes \delta_{\sim_Y})\circ \Delta_{X,Y}\mbox{ if }\sim=\sim_X\sqcup \sim_Y.
\end{align*}
Moreover, for any $x\in \bfP[\emptyset]$, 
\[(\varepsilon_\Delta\otimes \id)\circ \delta_{\sim_\emptyset}(x)=\varepsilon_\Delta(x)1_\bfP.\]
\end{itemize}

The following result is proved in \cite[Proposition 3.7]{Foissy41}:

\begin{prop}\label{prop1.1}
Let $(V,\cdot,\delta_V)$ be a counitary, not necessarily unitary bialgebra, and let $(\bfP,m,\Delta)$ be a twisted bialgebra with a contraction-extraction coproduct $\delta$. 
Then $\calF_V[\bfP]$ is a double bialgebra, with the product and the two coproducts defined as follows:
\begin{align*}
m(\overline{v_1\ldots v_m \otimes p}\otimes \overline{w_1\ldots w_n \otimes q})
&=\overline{\bfP[\sigma_{m,n}](m(p\otimes q))\otimes v_1\ldots v_mw_1\ldots w_n},\\
\Delta(\overline{v_1\ldots v_n\otimes p})
&=\sum_{I\subseteq [n]} 
\left(\overline{\prod_{i\in I} v_i\otimes \bfP[\sigma_I]\left(p^{(1)}_I\right)}\right)
\otimes \left(\overline{\prod_{i\notin I} v_i\otimes \bfP[\sigma_{[n]\setminus I}]\left(p^{(2)}_I\right)}\right),\\
\delta(\overline{v_1\ldots v_n\otimes p})&=\sum_{\sim \in \eq[n]}
\overline{\left(\prod_{\mbox{\scriptsize$C$  class of $\sim$}} \prod_{i\in C}^\cdot v'_i\right)\otimes p'_\sim}
\otimes 
\overline{v''_1\ldots v''_n\otimes p''_\sim},
\end{align*}
where $\sigma_{m,n}:[m]\sqcup [n]\longrightarrow [m+n]$ is the bijection sending the elements of $[m]$ to themselves and any element $k\in [n]$ to $k+m$, $\sigma_I:I\longrightarrow [|I|]$ is the unique increasing bijection (where $I$ is any subset of $[n]$), and, with Sweedler's notations 
\begin{align*}
\Delta_{I,[n]\setminus I}(p)&=p^{(1)}_I\otimes p^{(2)}_I,&\delta_\sim(p)&=p'_\sim \otimes p''_\sim,&\mbox{and}&&\delta_V(v)&=v'\otimes v''. 
\end{align*} \end{prop}

\section{The species of mixed graphs}

\subsection{Mixed graphs}

\begin{defi} \begin{enumerate}
\item A \emph{mixed graph} is a triple $G=(V(G),E(G),A(G))$ where:
\begin{enumerate}
\item $V(G)$ is a finite set, called the set of vertices of $G$,
\item $E(G)$ is a subset of $\{\{x,y\}\subseteq V(G),\: x\neq y\}$, called the set of edges of $G$,
\item $A(G)$ is a subset of $\{(x,y)\in V(G)^2,\:x\neq y\}$, called the set of arcs of $G$,
\end{enumerate}
such that, for any $x,y\in V(G)$, with $x\neq y$,
\begin{align*}
\{x,y\}\in E(G)&\Longrightarrow (x,y)\notin A(G)\mbox{ and } (y,x)\notin A(G),\\
(x,y)\in A(G)&\Longrightarrow \{x,y\}\notin E(G).
\end{align*}
\item For any finite set $X$, we denote by $\calG[X]$ the set of mixed graphs $G$ such that $V(G)=X$. 
The vector space generated by $\calG[X]$ is denoted by $\bfG[X]$.
This defines a species $\bfG$ of mixed graphs. 
\item A mixed graph is an \emph{oriented graph} if $E(G)=\emptyset$: this defines a subset of $\calG[X]$
denoted by $\calG_o[X]$ for any finite set $X$, and a subspecies of $\bfG$ denoted by $\bfG_o$.
\item A mixed graph $G$ is a \emph{simple graph} if $A(G)=\emptyset$: this defines a subset of $\calG[X]$
denoted by $\calG_s[X]$ for any finite set $X$, and a subspecies of $\bfG$ denoted by $\bfG_s$. 
\end{enumerate}\end{defi}

\begin{example} \label{ex1.1}
There are five elements in $\calG[2]$, which we graphically represent on the right:
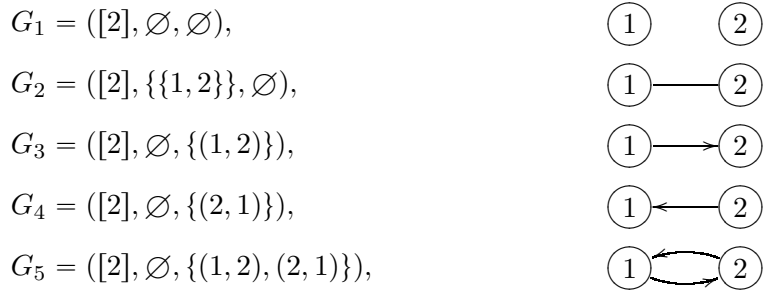
\begin{align*}
G_1&=([2],\emptyset,\emptyset),&\xymatrix{\rond{1}&\rond{2}}\\
G_2&=([2],\{\{1,2\}\},\emptyset),&\xymatrix{\rond{1}\ar@{-}[r]&\rond{2}}\\
G_3&=([2],\emptyset,\{(1,2)\}),&\xymatrix{\rond{1}\ar[r]&\rond{2}}\\
G_4&=([2],\emptyset,\{(2,1)\}),&\xymatrix{\rond{1}&\rond{2}\ar[l]}\\
G_5&=([2],\emptyset,\{(1,2),(2,1)\}),&\xymatrix{\rond{1}\ar@/_/[r]&\rond{2}\ar@/_/[l]}
\end{align*}
Moreover,
\begin{align*}
\calG_o[2]&=\{G_1,G_3,G_4,G_5\},& \calG_s[2]&=\{G_1,G_2\}. 
\end{align*}
\end{example}

\begin{remark}
In a mixed graph $G$, for any pair of vertices $\{x,y\}$ of $G$, there are five possibilities to define
edges or arcs between $x$ and $y$. Hence, if $X$ is of cardinality $n$, then 
\[|\calG[X]|=5^{\frac{n(n-1)}{2}}.\]
Similarly,
\begin{align*}
|\calG_o[X]|&=4^{\frac{n(n-1)}{2}},&|\calG_s[X]|&=2^{\frac{n(n-1)}{2}}.
\end{align*}
These gives respectively  sequences \href{https://oeis.org/A109345}{A109345},
\href{https://oeis.org/A053763}{A053763} and
\href{https://oeis.org/A006125}{A006125} of the OEIS \cite{Sloane}. 
\end{remark}

\begin{notation}
Let $G$ be a mixed graph and $x,y\in V(G)$. We shall write $x\arc{G} y$ if $(x,y)\in A(G)$
and $x \arete{G} y$ if $\{x,y\}\in E(G)$. 
\end{notation}

\begin{defi}
Let $G$ be a mixed graph. 
\begin{enumerate}
\item An oriented path in $G$ is a finite sequence $P=(x_0,\ldots,x_n)$ of vertices of $G$ such that for any $i\in \{0,\ldots,n-1\}$,
$x_i\arc{G} x_{i+1}$. The vertices $x_0$ and $x_n$ are respectively the beginning and the end of $P$.
\item A mixed path in $G$ is a finite sequence $P=(x_0,\ldots,x_n)$ of vertices of $G$ such that for any $i\in \{0,\ldots,n-1\}$,
$x_i\arc{G} x_{i+1}$ or $x_i \arete{G} x_{i+1}$. The vertices $x_0$ and $x_n$ are the extremities of $P$.
\item A path in $G$ is a finite sequence $P=(x_0,\ldots,x_n)$ of vertices of $G$ such that for any $i\in \{0,\ldots,n-1\}$,
$x_i \arc{G} x_{i+1}$ or $x_{i+1}\arc{G} x_i$ or $x_i\arete{G} x_{i+1}$. 
The vertices $x_0$ and $x_n$ are the extremities of $P$.
The mixed graph $G$ is connected if for any vertices $x,y\in V(G)$, there exists a path with extremities $x$ and $y$. 
\end{enumerate}
\end{defi}

\begin{example}
Let us consider the mixed graphs of Example \ref{ex1.1}. The connected ones are $G_2$, $G_3$, $G_4$ and $G_5$. 
\end{example}

\subsection{The twisted bialgebra of mixed graphs}

We now define a product and a coproduct on the species $\bfG$.
Let $X,Y$ be two finite sets, $G\in \calG[X]$ and $H\in \calG[Y]$. The mixed graph $GH\in \calG[X\sqcup Y]$, disjoint union of $G$ and $H$, is defined by
\begin{align*}
V(GH)&=V(G)\sqcup V(H),&
E(GH)&=E(G)\sqcup E(H),&
A(GH)&=A(G)\sqcup A(H).
\end{align*}
This product is bilinearly extended to $\bfG$. it is clearly associative, commutative and has a unit, which is the empty graph $1\in \calG[\emptyset]$. 
With this product, $\bfG$ is a commutative twisted algebra. Note that $\bfG_s$ and $\bfG_o$ are subalgebras of $\bfG$. 

\begin{defi}
Let $G$ be a mixed graph and let $I\subseteq V(G)$. 
\begin{enumerate}
\item The mixed graph $G_{\mid I}$ is defined by
\begin{align*}
V(G_{\mid I})&=I,\\
E(G_{\mid I})&=\{\{x,y\}\in E(G)\mid x,y\in I\},\\
A(G_{\mid I})&=\{(x,y)\in A(G)\mid x,y\in I\}.
\end{align*}
\item We shall say that $I$ is an ideal of $G$ if 
\begin{align*}
&\forall x,y\in V(G),&(x\in I \mbox{ and }x\arc{G} y)\Longrightarrow (y\in I).
\end{align*}
\end{enumerate}
\end{defi}

\begin{prop}\label{prop1.4}
We define a coproduct $\Delta$ on the species $\bfG[X]$ in the following way:
for any finite sets $X$ and $Y$, for any mixed graph $G\in \calG[X\sqcup Y]$,
\[\Delta_{X,Y}(G)=\begin{cases}
G_{\mid X}\otimes G_{\mid Y}\mbox{ if $Y$ is an ideal of $G$},\\
0\mbox{ otherwise}.
\end{cases}\]
Together with the product defined earlier, this coproduct makes $\bfG$ a twisted bialgebra.
\end{prop}

\begin{proof}
Let us first prove the coassociativity of $\Delta$. Let $X$, $Y$ and $Z$ be disjoint  finite sets and $G\in \calG[X\sqcup Y\sqcup Z]$. 
Then
\begin{align*}
(\Delta_{X,Y}\otimes \id)\circ \Delta_{X\sqcup Y,Z}(G)&=\begin{cases}
G_{\mid X}\otimes G_{\mid Y}\otimes G_{\mid Z}\\
\hspace{1cm}\mbox{ if $Z$ is an ideal of $G$ and $Y$ is an ideal of $G_{\mid X\sqcup Y}$},\\
0\mbox{ otherwise}. 
\end{cases}\\
(\id\otimes \Delta_{Y,Z})\circ \Delta_{X,Y\sqcup Z}(G)&=\begin{cases}
G_{\mid X}\otimes G_{\mid Y}\otimes G_{\mid Z}\\
\hspace{1cm}\mbox{ if $Y\sqcup Z$ is an ideal of $G$ and $Z$ is an ideal of $G_{\mid Y\sqcup Z}$},\\
0\mbox{ otherwise}. 
\end{cases}
\end{align*}
Moreover,
\begin{align*}
&\mbox{$Z$ is an ideal of $G$ and $Y$ is an ideal of $G_{\mid X\sqcup Y}$}\\
&\Longleftrightarrow \left(\forall (x,y)\in V(G)^2,\: x\arc{G} y\Longrightarrow (x,y)\notin (Y\times X) 
\sqcup (Z\times X)\sqcup (Z\times Y)\right)\\
&\Longleftrightarrow \mbox{$Y\sqcup Z$ is an ideal of $G$ and $Z$ is an ideal of $G_{\mid Y\sqcup Z}$},
\end{align*}
so $(\Delta_{X,Y}\otimes \id)\circ \Delta_{X\sqcup Y,Z}=(\id\otimes \Delta_{Y,Z})\circ \Delta_{X,Y\sqcup Z}$.
As for any graph $G$, $\emptyset$ and $V(G)$ are ideals of $G$,
\begin{align*}
\Delta_{V(G),\emptyset}(G)&=G\otimes 1,&\Delta_{\emptyset,V(G)}(G)&=1\otimes G.
\end{align*}
So $\Delta$ is counitary, and the counit is given by
\[\varepsilon_\Delta:\left\{\begin{array}{rcl}
\bfG[\emptyset]&\longrightarrow&\K\\
1\mbox{ (empty graph)}&\longrightarrow&1.
\end{array}\right.\]

Let $G\in \calG[X]$, $H\in \calG[Y]$ and $X'$, $Y'$ be sets such that $X\sqcup Y=X'\sqcup Y'$.
In $GH$, there is no arc between any element of $X$ and any element of $Y$, nor between any element of $Y$ and any element of $X$.
Hence, the ideals of $GH$ are of the form $I\sqcup J$, where $I$ is an ideal of $G$ and $J$ is an ideal of $H$.
Therefore,
\begin{align*}
\Delta_{X',Y'}(GH)&=\begin{cases}
G_{\mid X\cap X'}H_{\mid Y\cap X'}\otimes G_{\mid X\cap Y'}H_{\mid Y\cap Y'}\\
\hspace{1cm}\mbox{ if $Y'\cap X$ is an ideal of $G$ and $Y'\cap Y$ is an ideal of $H$},\\
0\mbox{ otherwise}
\end{cases}\\
&=\Delta_{X\cap X',X\cap Y'}(G)\Delta_{Y\cap X',Y\cap Y'}(H). 
\end{align*}
Moreover, $\Delta_{\emptyset,\emptyset}(1)=1\otimes 1$. So $\Delta$ is a morphism of twisted algebras. \end{proof}

\begin{example} With the notations of Example \ref{ex1.1},
\[\begin{array}{|c|c|c|c|c|c|}
\hline &\Delta_{\{1,2\},\emptyset}&\Delta_{\{1\},\{2\}}&\Delta_{\{2\},\{1\}}&\Delta_{\emptyset,\{1,2\}}\\
\hline G_1&G_1\otimes 1&\xymatrix{\rond{1}}\otimes \xymatrix{\rond{2}}&
\xymatrix{\rond{2}}\otimes \xymatrix{\rond{1}}&1\otimes G_1\\
\hline G_2&G_2\otimes 1&\xymatrix{\rond{1}}\otimes \xymatrix{\rond{2}}&
\xymatrix{\rond{2}}\otimes \xymatrix{\rond{1}}&1\otimes G_2\\
\hline G_3&G_3\otimes 1&\xymatrix{\rond{1}}\otimes \xymatrix{\rond{2}}&
0&1\otimes G_3\\
\hline G_4&G_4\otimes 1&0&
\xymatrix{\rond{2}}\otimes \xymatrix{\rond{1}}&1\otimes G_4\\
\hline G_5&G_5\otimes 1&0&0&1\otimes G_5\\
\hline \end{array}\]
\end{example}

\subsection{Contraction-extraction on mixed graphs}

\begin{defi}
Let $G\in \calG[X]$ and $\sim\in \eq[X]$. 
\begin{enumerate}
\item We define a mixed graph $G\mid\sim\in \calG[X]$ by
\begin{align*}
V(G\mid\sim)&=V(G),\\
E(G\mid\sim)&=\{\{x,y\}\in E(G)\mid\: x\sim y\},\\
A(G\mid \sim)&=\{(x,y)\in A(G)\mid\: x\sim y\}.
\end{align*}
In other words, $G\mid \sim$ is obtained from $G$ by deleting all the edges or arcs whose extremities are not equivalent;
or equivalently, $G\mid\sim$ is the disjoint union of the restrictions of $G$ to the equivalence classes of $\sim$.
\item We define a mixed graph $G/\sim\in \calG[X/\sim]$ by
\begin{align*}
V(G/\sim)&=X/\sim,\\
E(G/\sim)&=\{\{\cl_\sim(x),\cl_\sim(y)\}\mid\: \{x,y\}\in E(G), (x,y)\notin A(G), \:(y,x)\notin A(G),\: \cl_\sim(x)\neq \cl_\sim(y)\},\\
A(G/\sim)&=\{(\cl_\sim(x),\cl_\sim(y))\mid \: (x,y)\in A(G), \: \cl_\sim(x)\neq \cl_\sim(y)\}.
\end{align*}
In other words, $G/\sim$ is obtained from $G$ by identifying the vertices according to $\sim$,
then deleting the loops created in the process and the redundant edges, giving priority to the oriented ones.
\item We shall say that $\sim\in \eq^c[G]$ if for any equivalence class $C$ of $\sim$, $G_{\mid C}$ is connected.
\end{enumerate}
\end{defi}

\begin{prop} \label{prop1.6}
We define a contraction-extraction coproduct $\delta$ on $\bfG$ as follows: for any finite set $X$,
for any $\sim\in \eq[X]$, for any $G\in \calG[X]$,
\[\delta_\sim(G)=\begin{cases}
G/\sim\otimes G\mid \sim \mbox{ if }\sim\in \eq^c[G],\\
0\mbox{ otherwise}.
\end{cases}\]
\end{prop}

\begin{proof}
The compatibility of $\delta$ with the species structure is clear. Let us prove the coassociativity of $\delta$. 
Let $X$ be a finite set, $\sim,\sim'\in \eq[X]$ and $G\in \calG[X]$.\\

If $\sim\leq \sim'$, let us prove that $\sim \in \eq^c[G/\sim']$ and $\sim'\in \eq^c[G]$ if, and only if,
$\sim'\in \eq^c[G\mid \sim]$ and $\sim \in \eq^c[G]$.

$\Longrightarrow$. Let $C'$ be a class of $\sim'$. As $\sim'\in \eq^c[G]$, it is a connected subgraph of $G$.
Moreover, as $\sim\leq \sim'$, all its elements are in the same class of $\sim$, so $G_{\mid C'}=(G|\sim)_{\mid C'}$:
as a consequence, $(G|\sim)_{\mid C'}$ is connected, so $\sim'\in \eq^c[G\mid \sim]$. 
Let $C$ be a class of $\sim$, and $x,y\in C$. As $\sim \in \eq^c[G/\sim']$, it is connected in $G/\sim'$: 
there exists a path in $G/\sim'$ from $\cl_{\sim'}(x)$ to $\cl_{\sim'}(y)$. Moreover, as $\sim'\in \eq^c[G]$,
each $\cl_{\sim'}(z)$ is a connected subgraph of $G$, so there is a path from $x$ to $y$ in $G$: $\sim \in \eq^c[G]$.

$\Longleftarrow$. Let $C$ be a class of $\sim$. As $\sim\in \eq^c[G]$, any of its class is a connected subgraph of $G$,
so by contraction is a connected subgraph of $G/\sim'$: $\sim \in \eq^c[G/\sim']$. Let $C'$ be a class of $\sim'$.
As $\sim'\in \eq^c[G\mid \sim]$, it is a connected subgraph of $G\mid \sim$, so also of $G$: $\sim \in \eq^c[G/\sim']$.\\

As a conclusion,
\begin{align*}
(\delta_\sim\otimes \id)\circ \delta_{\sim'}(G)&=\begin{cases}
(G/\sim')/\sim\otimes (G/\sim')\mid \sim\otimes G\mid \sim'\mbox{ if $\sim \in \eq^c[G/\sim']$ and $\sim'\in \eq^c[G]$},\\
0\mbox{ otherwise}
\end{cases}\\
&=\begin{cases}
G/\sim\otimes (G\mid \sim)/\sim'\otimes (G\mid \sim)\mid \sim'
\mbox{ if $\sim'\in \eq^c[G\mid \sim]$ and $\sim \in \eq^c[G]$},\\
0\mbox{ otherwise}
\end{cases}\\
&=(\id \otimes \delta_{\sim'})\circ \delta_\sim(G).
\end{align*}

If we do not  have $\sim\leq \sim'$, then at least one class $C$ of $\sim$ intersects two classes of $\sim'$,
so intersects two connected components of $\mid \sim'$: we obtain that $\sim\notin \eq^c[G\mid \sim']$. So
$\delta_{\sim}(G\mid \sim')=0$ and finally $(\id \otimes \delta_\sim)\circ \delta_{\sim'}(G)=0$.\\

Let us now study the counit. 
We define a species morphism $\epsilon_\delta:\bfG\longrightarrow\com$ as follows: if $G\in \calG[X]$, 
\[\epsilon_\delta[X](G)=\begin{cases}
1\mbox{ if }E(G)=A(G)=\emptyset,\\
0\mbox{ otherwise.}
\end{cases}\]
Let $G\in \calG[X]$ and $\sim \in \eq[X]$. If $\sim$ is the equality of $X$,
then $\sim\in \eq^c[G]$, $G/\sim=G$ and $G\mid \sim$ as no edge, so $(\id \otimes \epsilon_\delta[X])(G)=G$.
Otherwise, either $G\notin \eq^c[G]$ or at least one class of $\sim$ contains an edge or an arc, so 
$\epsilon_\delta[X](G\mid \sim)=0$. In both cases, $(\id \otimes \epsilon_\delta[X])(G)=0$.

Let $\sim\in \eq^c[G]$, such that $E(G/\sim)=E(G\mid \sim)=\emptyset$. If two vertices of $G$ are related by an edge of an arc,
there are necessarily equivalent, so any connected component of $G$ is included in a single class of $\sim$.
As the classes of $\sim$ are connected, $\sim$ is the relation $\sim_c$ whose classes are the connected components of $G$.
Moreover, $G/\sim_c$ has no edge nor arc, and $G\mid \sim_c=G$. Therefore,
\begin{align*}
\sum_{\sim\in \eq[X]}(\epsilon_\delta[X/\sim]\otimes \id)\circ \delta_\sim(G)
&=\sum_{\sim\in \eq^c[G]}(\epsilon_\delta[X/\sim]\otimes \id)\circ \delta_\sim(G)\\
&=(\epsilon_\delta[X/\sim]\otimes \id)\circ \delta_{\sim_c}(G)\\
&=G\mid \sim_c\\
&=G.
\end{align*}

Let us prove the compatibility of $\delta$ with the algebraic structure. Obviously, $\delta_{\sim_\emptyset}(1)=1\otimes 1$.  
Let $X$ and $Y$ be two finite sets, $\sim\in \eq[X\sqcup Y]$,
$G\in \calG[X]$ and $H\in \calG[Y]$. If $\sim\neq \sim_X\sqcup \sim_Y$, at least one class $C$ of $\sim$ intersects
both $X$ and $Y$, so is not connected in $GH=m_{X,Y}(G\otimes H)$. Therefore, $\sim\notin \eq^c[GH]$ and
\[\delta_\sim\circ m_{X,Y}(G\otimes H)=0.\]
Let us assume that  $\sim=\sim_X\sqcup \sim_Y$. Then $\sim\in \eq^c[GH]$ if, and only if, $\sim_X\in \eq^c[G]$
and $\sim_Y \in \eq^c[H]$, as the connected components of $GH$ are the connected components of $G$ and of $H$.
If so, $(GH)/\sim=(G/\sim_X)(H/\sim_Y)$ and $(GH)\mid \sim=(G\mid \sim_X)(H\mid \sim_Y)$. Therefore,
\begin{align*}
\delta_\sim\circ m_{X,Y}(G\otimes H)&=\begin{cases}
(GH)/\sim\otimes (GH)\mid \sim\mbox{ if }\sim\in \eq^c[GH],\\
0\mbox{ otherwise}
\end{cases}\\
&=\begin{cases}
(G/\sim_X)(H/\sim_Y)\otimes (G\mid \sim_X)(H\mid \sim_Y)\mbox{ if }\sim_X\in \eq^c[G]\mbox{ and }\sim_Y \in \eq^c[H],\\
0\mbox{ otherwise}
\end{cases}\\
&=(m_{X/\sim_X,Y/\sim_Y}\otimes m_{X,Y})\circ (\id \otimes c\otimes \id)\circ (\delta_{\sim_X}\otimes \delta_{\sim_Y})
(G\otimes H).
\end{align*}

Let us finally prove the compatibility of $\delta$ with the coalgebraic structure. Obviously, 
\[(\varepsilon_\Delta \otimes \id)\circ \delta_{\sim_\emptyset}(1)=1=\varepsilon_\Delta(1)1.\]
 Let $X$ and $Y$ be two finite sets, $\sim_X\in \eq[X]$, $\sim_Y\in \eq[Y]$ and $G\in \calG[X]$. We put $\sim=\sim_X\sqcup \sim_Y$.  
\begin{align*}
(\Delta_{X/\sim_X,Y/\sim_Y}\otimes \id)\circ \delta_{\sim}(G)
&=\begin{cases}
(G/\sim)_{\mid X/\sim_X}\otimes (G/\sim)_{\mid Y/\sim_Y}\\
\hspace{.5cm}\mbox{ if $\sim\in \eq^c[G]$ and $Y/\sim_Y$ is an ideal of $G/\sim$},\\
0\mbox{ otherwise},
\end{cases}\\
m_{1,3,24}\circ (\delta_{\sim_X}\otimes \delta_{\sim_Y})\circ \Delta_{X,Y}(G)
&=\begin{cases}
(G_{\mid X})/\sim_X\otimes (G_{\mid Y})/\sim_Y\\
\hspace{.5cm}\mbox{ if $Y$ is an ideal of $G$, $\sim_X\in \eq^c[G_{\mid X}]$
and $\sim_Y\in \eq^c[G_{\mid Y}]$},\\
0\mbox{ otherwise},
\end{cases}
\end{align*}
Let us prove that $\sim\in \eq^c[G]$ and $Y/\sim_Y$ is an ideal of $G/\sim$ if, and only if, 
$Y$ is an ideal of $G$, $\sim_X\in \eq^c[G_{\mid X}]$ and $\sim_Y\in \eq^c[G_{\mid Y}]$.

$\Longrightarrow$. Let $y\in Y$ and $z\in X\sqcup Y$ such that $x\arc{G} y$. 
Then either $\cl_\sim(y)=\cl_\sim(z)$ or $\cl_\sim(y) \stackrel{_{G/\sim}}{\rightarrow}\cl_\sim(z)$.
As $Y/\sim_Y$ is an ideal of $G/\sim$, in both cases $z\in Y$. As $\sim=\sim_X\sqcup \sim_Y$,
its classes are the classes of $\sim_X$ and $\sim_Y$, and are connected by hypothesis. 
So  $\sim_X\in \eq^c[G_{\mid X}]$ and $\sim_Y\in \eq^c[G_{\mid Y}]$.

$\Longleftarrow$. As $\sim=\sim_X\sqcup \sim_Y$, its classes are the classes of $\sim_X$ and $\sim_Y$, which are connected
by hypothesis. Hence, $\sim\in \eq^c[G]$. Let $\cl_{\sim}(y)\in Y/\sim_Y$ and $\cl_\sim(z) \in [X\sqcup Y]/\sim$,
such that $\cl_\sim(y) \stackrel{_{G/\sim}}{\rightarrow}\cl_\sim(z)$. There exist $y',z'\in X\sqcup Y$
such that $y\sim y'$, $z\sim z'$ and $y'\arc{G} z'$. As $\sim=\sim_X\sqcup \sim_Y$, $y'\in Y$.
As $Y$ is an ideal of $G$, necessarily $z'\in Y$. As $\sim=\sim_X\sqcup \sim_Y$, $z'\in Y$ and finally
$\cl_\sim(z')\in Y/\sim_Y$.

Moreover,
\begin{align*}
(G/\sim)_{\mid X/\sim_X}&=(G_{\mid X})/\sim_X,&(G/\sim)_{\mid Y/\sim_Y}&=(G_{\mid Y})/\sim_Y,
\end{align*}
which finally proves the compatibility between $\delta$ and $\Delta$. \end{proof}

As a consequence, by Proposition \ref{prop1.1}, for any vector space $V$, we obtain a graded bialgebra $\calF_V[\bfG]$.
This is the vector space of mixed graphs whose vertices are decorated by elements of $V$,
any graph being linear in any of its decorations: these objects will be called $V$-linearly decorated graphs. 
For example, if $v_1,v_2,w_1,w_2\in V$ and $\lambda_1,\lambda_2,\mu_1,\mu_2\in \K$, in $\calF_V[\bfG]$,
if $v=\lambda_1 v_1+\lambda_2 v_2$ and $w=\mu_1 w_1+\mu_2 w_2$,
\begin{align*}
\xymatrix{\rond{v}\ar@/_/[r]&\rond{w}\ar@/_/[l]}
&=\lambda_1\mu_1\: \xymatrix{\rond{v_1}\ar@/_/[r]&\rond{w_1}\ar@/_/[l]}
+\lambda_2\mu_1 \:\xymatrix{\rond{v_2}\ar@/_/[r]&\rond{w_1}\ar@/_/[l]}\\
&+\lambda_1\mu_2\: \xymatrix{\rond{v_1}\ar@/_/[r]&\rond{w_2}\ar@/_/[l]}
+\lambda_2\mu_2 \:\xymatrix{\rond{v_2}\ar@/_/[r]&\rond{w_2}\ar@/_/[l]}.
\end{align*}
If $\mathcal{B}$ is a basis of $V$, a basis of $\calF_V[\bfG]$ is the set of mixed graphs whose vertices are decorated
by elements of $\mathcal{B}$. 
The product is the disjoint union. For any $V$-linearly decorated graph $G$,
\[\Delta(G)=\sum_{\mbox{\scriptsize $I$ ideal of $G$}} G_{\mid V(G)\setminus I}\otimes G_{\mid I}.\]

\begin{example}
If $v,w\in V$,
\begin{align*}
\Delta(\xymatrix{\rond{v}}\xymatrix{\rond{w}})
&=\xymatrix{\rond{v}}\xymatrix{\rond{w}}\otimes 1+1\otimes \xymatrix{\rond{v}}\xymatrix{\rond{w}}
+\xymatrix{\rond{v}}\otimes \xymatrix{\rond{w}}+\xymatrix{\rond{w}}\otimes \xymatrix{\rond{v}},\\
\Delta(\xymatrix{\rond{v}\ar@{-}[r]&\rond{w}})
&=\xymatrix{\rond{v}\ar@{-}[r]&\rond{w}}\otimes 1+1\otimes \xymatrix{\rond{v}\ar@{-}[r]&\rond{w}}
+\xymatrix{\rond{v}}\otimes \xymatrix{\rond{w}}+\xymatrix{\rond{w}}\otimes \xymatrix{\rond{v}},\\
\Delta(\xymatrix{\rond{v}\ar[r]&\rond{w}})&=\xymatrix{\rond{v}\ar[r]&\rond{w}}\otimes 1
+1\otimes \xymatrix{\rond{v}\ar[r]&\rond{w}}+\xymatrix{\rond{v}}\otimes \xymatrix{\rond{w}},\\
\Delta(\xymatrix{\rond{v}\ar@/_/[r]&\rond{w}\ar@/_/[l]})&=\xymatrix{\rond{v}\ar@/_/[r]&\rond{w}\ar@/_/[l]}
\otimes 1+1\otimes \xymatrix{\rond{v}\ar@/_/[r]&\rond{w}\ar@/_/[l]}.
\end{align*} \end{example}

The counit $\varepsilon_\Delta$ sends any mixed graph $G\neq 1$ to $0$. 
If $(V,\cdot,\Delta)$ is a  not necessarily unitary, commutative and cocommutative bialgebra, 
then $\calF_V[\bfG]$ inherits a second coproduct $\delta$: if $G$ is a $V$-linearly decorated graph,
\[\delta(G)=\sum_{\sim\in \eq^c[G]} G/\sim\otimes G\mid \sim,\]
where the vertices of $G/\sim \otimes G\mid \sim$ are decorated in the following way:
denoting by $d_G(x)$ the decoration of the vertex $x\in V(G)$, any vertex $\cl_\sim(x)$ of $G/\sim$
is decorated by the products of elements $d_G(y)'$, where $y\in \cl_\sim(x)$, whereas the vertex $x\in V(G\mid \sim)=V(G)$
is decorated by $d_G(x)''$, and everything being extended by multilinearity of each decoration. 
The counit $\epsilon_\delta$ is given on any mixed graph $G$ by
\[\epsilon_\delta(G)=\begin{cases}
\displaystyle \prod_{x\in V(G)} \epsilon_V\circ d_G(x) \mbox{ if }A(G)=E(G)=\emptyset,\\
0\mbox{ otherwise}.
\end{cases}\]
This construction is functorial in $V$.

\begin{example} If $v,w\in V$, with Sweedler's notation $\delta_V(u)=u'\otimes u''$, 
\begin{align*}
\delta(\xymatrix{\rond{v}}\xymatrix{\rond{w}})&=\xymatrix{\rond{v'}}\xymatrix{\rond{w'}}\:\otimes \:\xymatrix{\rond{v''}}\xymatrix{\rond{w''}},\\
\delta(\xymatrix{\rond{v}\ar@{-}[r]&\rond{w}})&=\xymatrix{\rond{v'}\ar@{-}[r]&\rond{w''}}\:
\otimes\:\xymatrix{\rond{v''}}\xymatrix{\rond{w''}}+\xymatrix{\rond{v'\cdot w'}}\otimes \:\xymatrix{\rond{v''}\ar@{-}[r]&\rond{w''}},\\[3mm]
\delta(\xymatrix{\rond{v}\ar[r]&\rond{w}})&=\xymatrix{\rond{v'}\ar[r]&\rond{w'}}\:\otimes\: \xymatrix{\rond{v''}}\xymatrix{\rond{w''}}
+\xymatrix{\rond{v'\cdot w'}}\otimes\:\xymatrix{\rond{v''}\ar[r]&\rond{w''}},\\[3mm]
\delta(\xymatrix{\rond{v}\ar@/_/[r]&\rond{w}\ar@/_/[l]})&=
\xymatrix{\rond{v'}\ar@/_/[r]&\rond{w'}\ar@/_/[l]}\:\otimes\: \xymatrix{\rond{v''}}\xymatrix{\rond{w''}}
+\xymatrix{\rond{v'\cdot w'}}\otimes\: \xymatrix{\rond{v''}\ar@/_/[r]&\rond{w''}\ar@/_/[l]}.
\end{align*}\end{example}

\begin{example}
We shall often work with $V=\K$, with its usual bialgebraic structure defined by $\delta_\K(1)=1\otimes 1$.
We shall then identify  any $V$-decorated mixed graph  whose any vertex is decorated by $1$ with the underlying mixed graph.
The double bialgebra $\calF_V[\bfG]$ is identified with $\calF[\bfG]$ and has for basis the set of (isomorphism classes of) mixed graphs. The coproduct simplifies. In order to improve the  readability, we shall write $\grdeuxoo$ if there are two arcs of opposite directions between two vertices.
Examples of coproducts $\Delta$ and $\delta$ are given in Tables \ref{table1} and \ref{table2}.
\end{example}

\begin{table}\label{table1}
\begin{align*}
\Delta(\grdeuxvide)
&=\grdeuxvide\otimes 1+1\otimes \grdeuxvide
+2\grun\otimes \grun,\\
\Delta(\grdeux)
&=\grdeux\otimes 1+1\otimes \grdeux
+2\grun\otimes \grun,\\
\Delta(\grdeuxo)&=\grdeuxo\otimes 1
+1\otimes \grdeuxo+\grun\otimes \grun,\\
\Delta(\grdeuxoo)&=\grdeuxoo
\otimes 1+1\otimes \grdeuxoo,\\
\Delta\left(\grtrois\right)&=\grtrois\otimes 1+1\otimes \grtrois+2\grdeux\otimes \grun+\grdeuxvide\otimes \grun+2\grun\otimes \grdeux+\grun\otimes\grdeuxvide,\\
\Delta\left(\grtroisoe\right)&=\grtroisoe\otimes 1+1\otimes \grtroisoe
+\grdeuxvide\otimes\grun+\grdeuxo\otimes\grun+\grun\otimes\grdeux+\grun\otimes \grdeuxo,\\
\Delta\left(\grtroiseo\right)&=\grtroiseo\otimes 1+1\otimes \grtroiseo
+\grdeuxo\otimes\grun+\grdeux\otimes\grun+\grun\otimes \grdeuxo+\grun\otimes\grdeuxvide,\\
\Delta\left(\grtroisooo\right)&=\grtroisooo\otimes 1+1\otimes \grtroisooo
+\grdeuxoo\otimes\grun,\\
\Delta\left(\grcompletun\right)&=\grcompletun\otimes 1+1\otimes\grcompletun
+3\grdeux\otimes\grun+3\grun\otimes\grdeux,\\
\Delta\left(\grcompletdeux\right)&=\grcompletdeux\otimes 1+1\otimes\grcompletdeux+\grdeuxo\otimes\grun+\grdeux\otimes\grun+\grun\otimes\grdeux+\grun\otimes\grdeuxo,\\
\Delta\left(\grcomplettrois\right)&=\grcomplettrois\otimes 1+1\otimes\grcomplettrois+2\grdeuxo\otimes\grun+\grun\otimes\grdeux,\\
\Delta\left(\grcompletquatre\right)&=\grcompletquatre\otimes 1+1\otimes\grcompletquatre+\grdeuxo\otimes\grun+\grun\otimes\grdeuxo,\\
\Delta\left(\grcompletcinq\right)&=\grcompletcinq\otimes 1+1\otimes\grcompletcinq,\\
\Delta\left(\grcompletsix\right)&=\grcompletsix\otimes 1+1\otimes\grcompletsix
++\grdeuxo\otimes\grun+\grun\otimes\grdeuxo.
\end{align*}
\caption{Examples of coproducts $\Delta$}
\end{table}

\begin{table}\label{table2}
\begin{align*}
\delta(\grdeuxvide)&=\grdeuxvide\otimes \grdeuxvide,\\
\delta(\grdeux)&=\grdeux
\otimes\grdeuxvide+\grun\otimes \grdeux,\\
\delta(\grdeuxo)&=\grdeuxo\otimes \grdeuxvide
+\grun\otimes\grdeuxo,\\
\delta(\grdeuxoo)&=
\grdeuxoo\otimes \grdeuxvide
+\grun\otimes \grdeuxoo,\\
\delta\left(\grtrois\right)&=\grtrois\otimes \grun\grun\grun+\grun \otimes \grtrois+2\grdeux\otimes \grdeux \grun,\\
\delta\left(\grtroisoe\right)&=\grtroisoe\otimes \grun\grun\grun+\grun \otimes \grtroisoe+\grdeuxo\otimes\grdeux \grun+\grdeux\otimes \grdeuxo \grun,\\
\delta\left(\grtroiseo\right)&=\grtroiseo\otimes \grun\grun\grun+\grun \otimes \grtroiseo+\grdeuxo\otimes\grdeux \grun+\grdeux\otimes \grdeuxo \grun,\\
\delta\left(\grtroisooo\right)&=\grtroisoo\otimes \grun\grun\grun+\grun \otimes \grtroisoo+\grdeuxo\otimes\grdeuxoo \grun+\grdeuxoo\otimes \grdeuxo \grun,\\
\delta\left(\grcompletun\right)&=\grcompletun\otimes \grun\grun\grun+\grun\otimes\grcompletun+3\grdeux\otimes\grdeux\grun,\\
\delta\left(\grcompletdeux\right)&=\grcompletdeux\otimes \grun\grun\grun+\grun\otimes\grcompletdeux+2\grdeuxo\otimes\grdeux\grun+\grdeux\otimes\grdeuxo\grun,\\
\delta\left(\grcomplettrois\right)&=\grcomplettrois\otimes \grun\grun\grun+\grun\otimes\grcomplettrois+2\grdeuxo\otimes\grdeuxo\grun+\grdeuxo\otimes\grdeux\grun,\\
\delta\left(\grcompletquatre\right)&=\grcompletquatre\otimes \grun\grun\grun+\grun\otimes\grcompletquatre+2\grdeuxo\otimes\grdeuxo\grun+\grdeuxoo\otimes\grdeux\grun,\\
\delta\left(\grcompletcinq\right)&=\grcompletcinq\otimes \grun\grun\grun+\grun\otimes\grcompletcinq+3\grdeuxoo\otimes\grdeuxo\grun,\\
\delta\left(\grcompletsix\right)&=\grcompletsix\otimes \grun\grun\grun+\grun\otimes\grcompletsix+2\grdeuxo\otimes\grdeuxo\grun+\grdeuxoo\otimes\grdeuxo\grun.
\end{align*}
\caption{Examples of coproducts $\delta$}
\end{table}

\begin{remark}
If $V$ is finite-dimensional, then, considering the number of vertices, $\calF_V[\bfG]$ is a graded bialgebra, whose homogeneous components are finite-dimensional. 
The dimension of the homogeneous component of degree $n$ of $\calF_V[\bfG]$ is given by the number of isomorphism classes of mixed graphs whose vertices are decorated by elements of $[N]$,
where $N=\dim(V)$, see the Appendix for more details on these numbers.
\end{remark}

\begin{prop}
Let $V$ be a (non necessarily unitary) commutative and cocommutative bialgebra. 
For any linearly $V$-decorated mixed graph $G$, we denote by $d_G:V(G)\longrightarrow V$ the decoration map of $G$
and by $\overline{G}$ the underlying mixed graph. Then the following map is a double bialgebra morphism:
\[\Theta_V:\left\{\begin{array}{rcl}
\calF_V[\bfG]&\longrightarrow&\calF[\bfG]\\
G&\longmapsto&\displaystyle \left(\prod_{x\in V(G)} \epsilon_V\circ d_G(x)\right) \overline{G}.
\end{array}\right.\]
\end{prop}

\begin{proof}
The counit $\epsilon_V:V\longrightarrow \K$ is a bialgebra map. By functoriality (in $V$), $\Theta_V$ is a double bialgebra morphism. 
\end{proof}

\subsection{Cofreeness of the coalgebra $\calF[\bfG]$}

Let us define a second product on $\calF[\bfG]$.

\begin{prop}
Let $G$ and $H$ be two mixed graphs. The mixed graph $G\curvearrowright H$ is defined by
\begin{align*}
V(G\curvearrowright H)&=V(G)\sqcup V(H),\\
E(G\curvearrowright H)&=E(G)\sqcup E(H),\\
A(G\curvearrowright H)&=A(G)\sqcup A(H)\sqcup (V(G)\times V(H)).
\end{align*}
This product is bilinearly extended to $\calF[\bfG]$. Then $(\calF[\bfG],\curvearrowright,\Delta)$
is a unital infinitesimal bialgebra in the sense of \cite[Definition 2.1]{Loday2006}.
\end{prop}

\begin{proof}
As we already know that $\Delta$ is coassociative and unitary, it remains to prove that:
\begin{enumerate}
\item $\curvearrowright$ is associative and unitary.
\item For any $x,y\in \calF[\bfG]$, 
$\Delta(x\curvearrowright y)=(x\otimes 1)\curvearrowright \Delta(y)+\Delta(x)\curvearrowright(1\otimes y)-x\otimes y$.
\end{enumerate}
1. Let $G,H,K$ be three mixed graphs. Then 
\begin{align*}
V((G\curvearrowright H)\curvearrowright K)=V(G\curvearrowright (H\curvearrowright K))
&=V(G)\sqcup V(H) \sqcup V(K),\\
E((G\curvearrowright H)\curvearrowright K)=E(G\curvearrowright (H\curvearrowright K))
&=E(G)\sqcup E(H) \sqcup E(K),\\
A((G\curvearrowright H)\curvearrowright K)=A(G\curvearrowright (H\curvearrowright K))
&=A(G)\sqcup A(H) \sqcup A(K) \sqcup (V(G)\times V(H))\\
&\sqcup (V(G)\times V(K))\sqcup (V(H)\sqcup V(K)),
\end{align*}
so $(G\curvearrowright H)\curvearrowright K)=G\curvearrowright (H\curvearrowright K)$.
Therefore, $\curvearrowright$ is associative. The unit is the empty mixed graph 1.\\

2. Let $G,H$ be two mixed graph. As there is an arc from any vertex of $G$ to any vertex of $H$ in $G\curvearrowright H$, 
the ideals of $G\curvearrowright H$ are:
\begin{itemize}
\item $I\sqcup V(H)$ where $I$ is an ideal of $G$. For such an ideal, 
\begin{align*}
(G\curvearrowright H)_{\mid I\sqcup V(H)}&=G_{\mid I}\curvearrowright H,&
(G\curvearrowright H)_{\mid V(G\curvearrowright H)\setminus ( I\sqcup V(H))}&=G_{\mid V(G) \setminus I}.
\end{align*}
\item Ideals $J$ of $H$. For such an ideal, 
\begin{align*}
(G\curvearrowright H)_{\mid J}&= H_{\mid J},&
(G\curvearrowright H)_{\mid V(G\curvearrowright H)\setminus J}&= G \curvearrowright H_{\mid V(H) \setminus J}.
\end{align*}
Note that the ideal $V(H)$ appears twice in this list, for $I=\emptyset$ and $J=V(H)$. Therefore,
\begin{align*}
\Delta(G\curvearrowright H)&=\sum_{\mbox{\scriptsize $I$ ideal of $G$}}
G_{\mid V(G) \setminus I}\otimes G_{\mid I}\curvearrowright H
+\sum_{\mbox{\scriptsize $J$ ideal of $G$}}
G \curvearrowright H_{\mid V(H) \setminus J}\otimes H_{\mid J}-G\otimes H\\
&=\Delta(G)\curvearrowright(1\otimes H)+(G\otimes 1)\curvearrowright \Delta(H)-G\otimes H. \qedhere
\end{align*}
\end{itemize}
\end{proof}

From \cite[Theorem 2.6]{Loday2006}:

\begin{cor}\label{cor1.9}
The coalgebra $(\calF[\bfG],\Delta)$ is isomorphic to the coalgebra $T(\mathrm{Prim}(\calF[\bfG]))$ with the 
deconcatenation coproduct. 
\end{cor}

\begin{remark}
The same proof can be adapted to any $\calF_V[\bfG]$. 
\end{remark}

\section{Sub-objects and quotients of mixed graphs}

\subsection{Simple and oriented graphs}

\begin{prop}\label{prop2.1}
$\bfG_s$ and $\bfG_o$ are twisted subbialgebras of $(\bfG,m,\Delta)$ and are stable under the contraction-extraction coproduct
$\delta$. 
\end{prop}

\begin{proof}
If $G$ and $H$ are simple graphs, then $GH$ is a simple graph. If $G$ is a simple graph, then all its subgraphs are also simple graphs. 
Moreover, if $\sim \in \eq^c[G]$, then $G/\sim$ and $G\mid \sim$ are also simple graphs. 
The proof is similar for oriented graphs.
\end{proof}

\begin{cor}\label{cor2.2}
For any vector space $V$, $\calF_V[\bfG_s]$ is a subbialgebra of $\calF_V[\bfG]$ and 
$\calF_V[\bfG_o]$ is a subbialgebra of $\calF_V[\bfG]$.
For any (non necessarily unitary) commutative and cocommutative bialgebra $V$, 
$\calF_V[\bfG_s]$ is a double subbialgebra of $\calF_V[\bfG]$ and  $\calF_V[\bfG_o]$ is a double subbialgebra of $\calF_V[\bfG]$.
\end{cor}

In particular, $\calF_\K[\bfG_s]=\calF[\bfG_s]$ is the double bialgebra of graphs of \cite{Foissy39,Foissy36}
and $\calF_\K[\bfG_o]=\calF[\bfG_o]$ is the double bialgebra of \cite{Calaque2011}.

\subsection{Acyclic mixed graphs and finite topologies}

Let us recall the following definition:

\begin{defi}
Let $X$ be a finite set. 
\begin{enumerate}
\item A topology on $X$ is a subset $\calO$ of the set of subsets of $X$ such that:
\begin{itemize}
\item $\emptyset$ and $X$ belong to $\calO$.
\item If $O_1$, $O_2\in \calO$, then $O_1\cap O_2\in \calO$ and $O_1\cup O_2\in \calO$.
\end{itemize}
\item The set of topologies on $X$ is denoted by $\caltopo[X]$ and the space generated by $\caltopo[X]$ is denoted by $\bftopo[X]$. This defines a species $\bftopo$. 
\item A topology $\calO$ is $T_0$ is for any $x,y\in X$, with $x\neq y$, there exists $O\in \calO$ such that $(x\in O$ and $y\notin O$) or ($x\notin O$ and $y\in O$).
\end{enumerate}
\end{defi}

\begin{example}
Let $G\in \calG[X]$ be a mixed graph. We denote by $\calO_G$ the set of ideals of $G$.
\end{example}

Let us prove this reformulation of Alexandroff's theorem  \cite{Alexandroff1937}: 

\begin{lemma} \label{lemmeAlex}
Let $X$ be a finite set and $\calO$ be a topology on $X$. There exists an oriented graph $G$ such that $\calO_G=\calO$.
\end{lemma}

\begin{proof}
Let $\calO\in \caltopo[X]$. We define a relation $\preceq$ on $X$ as follows: for any $x,y\in X$, $x\preceq y$ if any $O\in \calO$ containing $x$ also contains $y$. 
We then define an oriented graph $G$ by $V(G)=X$ and
\[E(G)=\{(x,y)\in X^2\mid x\neq y\mbox{ and }x\preceq y\}.\]
Let us prove that $\calO_G=\calO$. Let $O\in \calO$, $x\in O$ and $y\in X$ such that $x\arc{G} y$. Then $x\preceq y$: by definition of $\preceq$, $y\in O$. So $O\in \calO_G$.Therefore, $\calO\subseteq \calO_G$.
Let $O\in \calO_G$. As $O$ is an ideal of $G$,
\begin{align*}
O&=\{y\in X\mid \exists x\in O,\: x\preceq y\}\\
&=\bigcup_{x\in O}\{y\in X\mid x\preceq y\}\\
&=\bigcup_{x\in X}\left(\bigcap_{O'\in \calO,\: x\in O'}O'\right).
\end{align*}
As $\calO$ is a topology, $O\in \calO$, so $\calO=\calO_G$. 
\end{proof}

\begin{prop}
The species $\bftopo$ is equipped with a twisted bialgebra structure as follows:
\begin{itemize}
\item For any finite sets $X,Y$, for any $(\calO_X,\calO_Y)\in \caltopo[X]\times \caltopo[Y]$,
\[m_{X,Y}(O_X\times O_Y)=\{I\sqcup J,\: I\in \calO_X,\: J\in \calO_Y\}.\]
\item For any finite sets $X,Y$, for any $\calO\in \caltopo[X\sqcup Y]$,
\[\Delta_{X,Y}(\calO)=\begin{cases}
\calO_{\mid X}\otimes \calO_{\mid Y}\mbox{ if }Y\in \calO,\\
0\mbox{ otherwise},
\end{cases}\]
where
\begin{align*}
\calO_{\mid X}&=\{X\cap O\mid O\in \calO\},&\calO_{\mid Y}&=\{Y\cap O\mid O\in \calO\}.
\end{align*}
Moreover, the following map is a surjective morphism of twisted bialgebras:
\[\Upsilon:\left\{\begin{array}{rcl}
\bfG&\longrightarrow&\bftopo\\
G\in \calG[X]&\longrightarrow&\calO_G\in \caltopo[X].
\end{array}\right.\]
\end{itemize}\end{prop}

\begin{proof}
The map $\Upsilon$ is clearly a species morphism. By Lemma \ref{lemmeAlex}, it is surjective.
Let $G,G',H,H'$ be graphs such that $\calO_G=\calO_{G'}$ and $\calO_H=\calO_{H'}$. Then
\[\calO_{GH}=\calO_G\calO_H=\calO_{G'}\calO_{H'}=\calO_{G'H'}.\]
Therefore, the product of $\bfG$ is compatible with the products of $\bfG$ and $\bftopo$. \\

For any graph $G\in \calG[X]$ and for any $Y\subset X$, $(\calO_G)_{\mid Y}=\calO_{G_{\mid Y}}$.
This implies that $\upsilon$ is compatible with the coproducts of $\bfG$ and $\bftopo$. 
As $\Upsilon$ is surjective and $\bfG$ is a twisted bialgebra, $\bftopo$ is also a twisted bialgebra.
\end{proof}

This map $\Upsilon$ is not compatible with the contraction-extraction coproduct:
for example, if
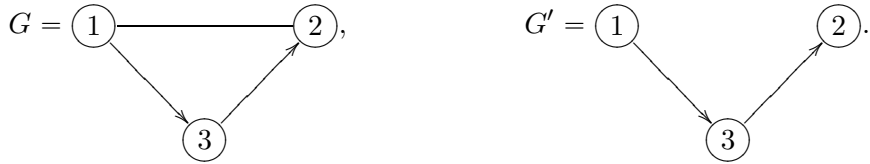
\begin{align*}
G&=\xymatrix{\rond{1}\ar@{-}[rr]\ar[rd]&&\rond{2}\\
&\rond{3}\ar[ru]&},&
G'&=\xymatrix{\rond{1}\ar[rd]&&\rond{2}\\
&\rond{3}\ar[ru]&}.
\end{align*}
then $\calO_{G_1}=\calO_{G_2}=\{\{1,2,3\},\{2,3\},\{2\},\emptyset\}$. Let us denote by $\sim$ the equivalence with classes $\{1,2\}$ and $\{3\}$. Then
\begin{align*}
\delta_\sim(G)&=\xymatrix{\rond{1,2}\ar@/_/[r]&\rond{3}\ar@/_/[l]}
\otimes \xymatrix{\rond{1}\ar@{-}[r]&\rond{2}}\xymatrix{\rond{3}},&
\delta_\sim(G')&=0.
\end{align*}

In order to obtain a second coproduct, we have to restrict ourselves to acyclic mixed graphs:

\begin{prop}
Let $G$ be a mixed graph. We shall say that $G$ is \emph{acyclic} if it does not contain any oriented path 
$x_0\arc{G} \ldots \arc{G} x_n$ with $x_0=x_n$ and $n\geqslant 2$. 
Acyclic mixed graphs and acyclic oriented graphs form two twisted subbialgebras of $(\bfG,m,\Delta)$.
\end{prop}

\begin{proof}
Obviously, if $G$ and $H$ are acyclic mixed graphs, then $GH$ is acyclic; if $G$ is acyclic and $I\subseteq V(G)$,
then $G_{\mid I}$ is also acyclic. Therefore, $\bfG_{ac}$ is a twisted subbialgebra of $\bfG$ and 
$\bfG_{aco}$ is a twisted subbialgebra of $\bfG_o$.
\end{proof}

\begin{remark}
$\bfG_{ac}$ is not stable under $\delta$. For example,  let us consider the following acyclic oriented graph:
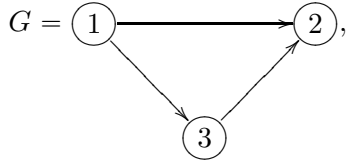
\begin{align*}
G&=\xymatrix{\rond{1}\ar[rr]\ar[rd]&&\rond{2}\\
&\rond{3}\ar[ru]&},
\end{align*}
Let us denote by $\sim$ the equivalence with classes $\{1,2\}$ and $\{3\}$. Then
\begin{align*}
\delta_\sim(G)&=\xymatrix{\rond{1,2}\ar@/_/[r]&\rond{3}\ar@/_/[l]}
\otimes \xymatrix{\rond{1}\ar[r]&\rond{2}}\xymatrix{\rond{3}}.
\end{align*} \end{remark}

\begin{prop}
A contraction-extraction coproduct on $\bfG_{ac}$ is defined as follows:
for any acyclic graph $G\in \calG_{ac}[X]$, for any $\sim \in \eq[X]$,
\[\delta_\sim(G)=\begin{cases}
G/\sim\otimes G\mid \sim\mbox{ if }\sim\in \eq^{c,ac}[G],\\
0\mbox{ otherwise.}
\end{cases}\]
where $\eq^{c,ac}[G]$ the set of equivalences on $V(G)$ such that the classes of $G$ are connected and $G/\sim$ is acyclic. 
Moreover, the following map is a surjective morphisms of twisted bialgebras, compatible with the contraction-extraction coproducts:
\begin{align}
\label{eqvarpi0}
\varpi_0&:\left\{\begin{array}{rcl}
\bfG&\longrightarrow&\bfG_{ac}\\
G\in \calG[X]&\longmapsto&\begin{cases}
G\mbox{ if $G$ is acyclic},\\
0\mbox{ otherwise}.
\end{cases}\end{array}\right.
\end{align}
\end{prop}

\begin{proof}
Let $\bfI$ be the subspecies of $\bfG$ of non acyclic mixed graphs. 
If $G$ is a non acyclic mixed graph, then for any mixed graph $H$, $GH$ is not acyclic: $\bfI$ is an ideal. 
If $I$ is an ideal of $G$, if it contains a vertex on a cycle of $G$, then it contains all the vertices of the cycle:
therefore, $G_{\mid I}$ or $G_{V(G)\setminus I}$ is not acyclic, which proves that $\bfI$ is a coideal for $\Delta$.
Let $\sim \in \eq^c[G]$. Let us consider a cycle $C$ of $G$. If all the vertices of $C$ are equivalent for $\sim$,
then $G\mid \sim$ contains a cycle, so is not acyclic. Otherwise, $G/\sim$ contains a cycle: $\bfI$ is a coideal for $\delta$.
Identifying the species $\bfG/\bfI$ and $\bfG_{ac}$ via $\varpi_0$, $\bfG_{ac}$ inherits a contraction-extraction coproduct $\delta$,
which is precisely the one defined in this proposition.
\end{proof}

Similarly, restricting $\varpi_0$ to $\bfG_o$, its image $\bfG_{aco}$ inherits a contraction-extraction coproduct $\delta$, as a sub-quotient of $\bfG$. 
The image of acyclic graphs by $\Upsilon$ is given by $T_0$-topologies:

\begin{defi}
Let $X$ be a finite set and $\calO$ a topology on $X$. We shall say that $\calO$ is $T_0$ if for any $x\neq y\in X$,
there exists $O\in \calO$ such that $(x\in O$ and $y\notin O$) or ($x\notin O$ and $y\in O$).
This defines a subset $\calpos[X]$ of $\caltopo[X]$ subspecies $\bfpos$ of $\bftopo$.
\end{defi}

\begin{lemma} \label{lemmeAlex2}
Let $X$ be a finite set and $G\in \calG[X]$. The topology $\calO_G$ is $T_0$ if, and only if, $G$ is acyclic.
$x_0\arc{G} \ldots \arc{G} x_n$ with $x_0=x_n$ and $n\geqslant 2$. 
\end{lemma}

\begin{proof}
Let us assume that $G$ has a cycle $x_0\arc{G} \ldots \arc{G} x_n$ with $x_0=x_n$ and $n\geqslant 2$.
Then any ideal of $G$ containing one of the $x_i$'s contains all the $x_i$'s, so $\calO_G$ is not $T_0$.
Let us assume that $G$ is acyclic. Let $x\neq y\in X$. We consider
\begin{align*}
O_x&=\{z\in X\mid \exists k\in \N,\: \exists x_1,\ldots,x_k\in X, \:x\arc{G}x_1\arc{G}\ldots \arc{G} x_k \arc{G}z\},\\
O_y&=\{z\in X\mid \exists k\in \N,\: \exists y_1,\ldots,y_k\in X, \:y\arc{G}y_1\arc{G}\ldots \arc{G} y_k \arc{G}z\}.
\end{align*}
Both $O_x$ and $O_y$ are ideals of $G$, so belong to $\calO$. Moreover, for $k=0$, $x\in O_x$ and $y\in O_y$. If $x\in O_y$ and $y\in O_x$, 
then there exists $k,l\in \N$ and $x_1,\ldots,x_k,y_1,\ldots,y_l\in X$ such that
\[:x\arc{G}x_1\ldots \arc{G} x_k \arc{G}y\arc{G}y_1\ldots \arc{G} y_l \arc{G}x,\]
which contradicts the acyclicity of $G$. So $y\notin O_x$ or $x\notin O_y$: $\calO_G$ is $T_0$.  
\end{proof}

\begin{lemma}\label{lemmaconnexe}
Let $G$ be an oriented graph. Then $\calO_G$ is connected if and only if $G$ is connected.
\end{lemma}

\begin{proof}
Let us assume that $\calO_G$ is not connected. Let $O_1,O_2\in \calO_G$, both nonempty, such that $V(G)=O_1\sqcup O_2$. If $x\arc{G} y$, with $x\in O_1$, then as $O_1$ is an ideal, $y\in O_1$.
Consequently, there is no arc from a vertex of $O_1$ to a vertex of $O_2$. Symmetrically, there is no arc from a vertex of $O_2$ to a vertex of $O_1$. So $G$ is not connected.
Let us assume that $G$ is not connected. We can write $V(G)=O_1\sqcup O_2$, such that there is no arc from a vertex of $O_1$ to a vertex of $O_2$, nor from  a vertex of $O_2$ to a vertex of $O_1$.
Consequently, $O_1,O_2\in \calO_G$, so $\calO_G$ is not connected.
\end{proof}

\begin{prop}
There exists a unique product,  a unique coproduct and a unique contraction-extraction coproduct on $\bfpos$ making the map
$\Upsilon_{\mid \bfG_{aco}}:\bfG_{aco}\longrightarrow\bfpos$ 
a morphism of twisted bialgebras, compatible with the  contraction-extraction coproduct.
\end{prop}

\begin{proof}
By Lemma \ref{lemmeAlex2}, $\Upsilon(\bfG_{aco})=\bfpos$, which gives the unicity of the contraction-extraction coproduct on $\bfpos$ compatible with $\Upsilon$.
The product and the coproduct on $\bfpos$ are obviously the restriction of the product and of the coproduct on finite topologies.\\

Let $X$ be a finite set, $G,G'\in \bfG_{aco}[X]$ such that $\Upsilon(G)=\Upsilon(G')$, and let $\sim\in \eq[X]$. Obviously, $\Upsilon(G\mid \sim)=\Upsilon(G'\mid \sim)$. Moreover,
\[\calO_{G/\sim}=\{\pi_\sim(O)\mid O\in \calO_G\}=\{\pi_\sim(O)\mid O\in \calO_{G'}\}=\calO_{G'/\sim},\]
where $\pi_\sim:X\longrightarrow X/\sim$ is the canonical surjection. So $\Upsilon(G/\sim)=\Upsilon(G'/\sim)$.
Let us now prove that $\eq^{c,ac}[G]=\eq^{c,ac}[G']$.
If $\sim \in \eq^{c,ac}[G]$, then $G/\sim$ is acyclic, so $\Upsilon(G/\sim)$ is $T_0$, by Lemma \ref{lemmeAlex2}, so $\Upsilon(G'/\sim)=\Upsilon(G/\sim)$ is $T_0$ and $G'/\sim$ is acyclic.
Moreover, by Lemma \ref{lemmaconnexe}, the connected components of $\Upsilon(G\mid \sim)=\Upsilon(G'\mid \sim)$ are the connected components of the oriented graph $G\mid \sim$,
that is to say the classes of $\sim$ as $\sim\in \eq^{c,ac}[G]$. Consequently, if $C$ is a class of $\sim$, $(\calO_{G'})_{\mid C}$ is connected: $\sim  \in \eq^{c,ac}[G']$. 
By symmetry, we obtain $\eq^{c,ac}[G]=\eq^{c,ac}[G']$. \\

As a consequence, for any $\sim\in \eq[X]$,
\begin{align*}
(\Upsilon \otimes \Upsilon)\circ \delta_\sim(G)&=
\begin{cases}
\Upsilon(G/\sim)\otimes \Upsilon(G\mid \sim)\mbox{ if }\sim\in \eq^{c,ac}[G],\\
0\mbox{ otherwise}
\end{cases}\\
&=
\begin{cases}
\Upsilon(G'/\sim)\otimes \Upsilon(G'\mid \sim)\mbox{ if }\sim\in \eq^{c,ac}[G'],\\
0\mbox{ otherwise}
\end{cases}\\
&=(\Upsilon \otimes \Upsilon)\circ \delta_\sim(G').
\end{align*}
Consequently, $\bfpos$ inherits a contraction-extraction coproduct as a quotient of $\bfG_{aco}$.  \end{proof}

\begin{remark}
Let $X$ be a finite set and $\calO$ be a topology on $X$. For any $x,y\in X$, we shall say that $x\sim_\calO y$ if any $O\in \calO$ containing $x$ or $y$ contains both $x$ and $y$. 
This defines an equivalence $\sim_\calO$ on $X$. Moreover, $X/\sim_\calO$ inherits from $\calO$ a topology, which turns out to be $T_0$. In other words, we obtain that
$\bftopo[X]$ and $\displaystyle \bigoplus_{\sim\in \eq[X]} \bfpos[X/\sim]$ are isomorphic, which gives a species isomorphism between $\bftopo$ and $\bfpos\circ \com$,
where here $\circ$ is the composition of species. 
The double twisted bialgebra structure which we obtain in this way is described in \cite{Foissy39}. Applying Aguiar and Mahajan's bosonic Fock functor \cite{Aguiar2010}, 
we obtain the double algebra of finite topologies of \cite{Foissy37}. 
\end{remark}

\subsection{Totally acyclic graphs}

\begin{defi}
Let $G$ be a mixed graph. We shall say that it is totally acyclic if does not contain any mixed path $(x_0,\ldots,x_n)$,
with $x_0=x_n$ and $n\geq 2$. Totally acyclic graphs form a subspecies $\bfG_{tac}$ of $\bfG$. 
\end{defi}

Note that totally acyclic mixed graphs are simply called acyclic in \cite{Beck2012}.

\begin{prop}
$\bfG_{tac}$ is a twisted subbialgebra of $\bfG$.
\end{prop}

\begin{proof}
If $G$ and $H$ are totally acyclic graphs, then $GH$ is totally acyclic. So $\bfG_{tac}$ is a twisted subalgebra of $\bfG$.
Let $G$ be a totally acyclic mixed graph and $I\subseteq V(G)$. As $G$ does not contain any mixed cycle, 
so does $G_{\mid I}$: $G_{\mid I}$ is totally acyclic. As a conclusion, $\bfG_{tac}$ is a twisted subcoalgebra of $\bfG$. 
\end{proof}

Consequently, for any vector space $V$, $\calF_V[\bfG_{tac}]$ is a subbialgebra of $(\calF_V[\bfG],m,\Delta)$.
The subspecies $\bfG_{tac}$ is not stable under $\delta$. For example, considering the mixed graph
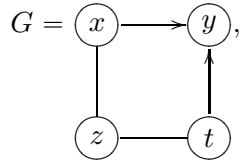
\[G=\xymatrix{\rond{x}\ar[r]\ar@{-}[d]&\rond{y}\\ \rond{z}\ar@{-}[r]&\rond{t}\ar[u]},\]
which is totally acyclic, the equivalence relation $\sim$ whose classes are $\{x,y\}$, $\{z\}$ and $\{t\}$ belongs to $\eq^c[G]$
(in fact, even to $\eq^{c,ac}[G]$), and 
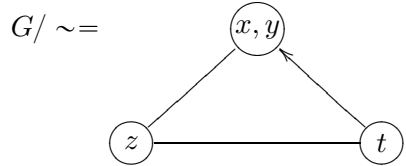
\[G/\sim=\xymatrix{&\rond{x,y}&\\
\rond{z}\ar@{-}[ru] \ar@{-}[rr]&&\rond{t}\ar[lu]}\]
 is clearly not totally acyclic.

\section{Applications}

\subsection{Three polynomial invariants}

Let $(V,\cdot,\delta_V)$ be a  non necessarily unitary, commutative bialgebra.
From \cite[Theorem 3.9]{Foissy40}, there exists a unique morphism $\phi_1$ of double bialgebras  from 
$(\calF_V[\bfG],m,\Delta,\delta)$ onto $(\K[X],m,\Delta,\delta)$ where the two coproducts of $\K[X]$ are defined by
\begin{align*}
\Delta(X)&=X\otimes 1+1\otimes X,&\delta(X)&=X\otimes X. 
\end{align*}
Let us determine $\phi_1$, firstly when $V=\K$. Let $G\in \calG[X]$, nonempty. Then, still by \cite[Theorem 3.9]{Foissy40},
\[\phi_1(G)=\sum_{k=0}^\infty \epsilon_\delta^{\otimes (k-1)}\circ \tdelta^{(k-1)}(G) H_k(X),\]
where $H_k$ is the $k$-th Hilbert polynomial:
\[H_k(X)=\frac{X(X-1)\ldots (X-k+1)}{k!}.\]

\begin{defi}
Let $G$ be a mixed graph.
\begin{enumerate}
\item A valid coloring of $G$ is a map $c:V(G)\longrightarrow \N_{>0}$ such that
\begin{align*}
&\forall x,y\in V(G),& 
\begin{cases}
x\arc{G} y&\Longrightarrow c(x)<c(y),\\
x\arete{G} y&\Longrightarrow c(x)\neq c(y).
\end{cases}
\end{align*}
\item A valid coloring $c$ of $G$ is packed if $c(V(G))=[\max(c)]$.
The set of valid packed colorings of $G$ is denoted by $\VPC(G)$.
\end{enumerate}
\end{defi}

\begin{prop}\label{prop3.2}
The unique morphism of double bialgebras from $\calF[\bfG]$ to $\K[X]$ is given 
on any mixed graph $G$ by
\begin{align*}
P_{chr_S}(G)&=\sum_{c \in \VPC(G)} H_{\max(c)}.
\end{align*}
Consequently, if $N\in \N$, $P_{chr_S}(G)(N)$ is the number of valid colorings $c$ such that $\max(c)\leq N$:
we recover the (strong) chromatic polynomial $P_{chr_S}(G)$ of \cite{Beck2012}. 
\end{prop}

\begin{proof}
For any $k\geqslant 1$, for any mixed graph $G$,  we denote by $L_k(G)$ the set of surjections $c:V(G)\longrightarrow [k]$ such that
\begin{align*}
&\forall x,y\in V(G),& x\arc{G}{y}\Longrightarrow c(x)\leq c(y).
\end{align*}
By definition of the coproduct $\Delta$, for any mixed graph $G$ with $n\geqslant 1$ vertices,
\begin{align*}
\tdelta^{(k-1)}(G)&=\sum_{c \in L_k(G)} G_{\mid c^{-1}(1)}\otimes \ldots \otimes G_{\mid c^{-1}(k)},
\end{align*}
and consequently, for any $V$-linearly decorated mixed graph $G$,
\begin{align*}
\epsilon_\delta^{\otimes (k-1)}\circ \tdelta^{(k-1)}(G)&=|\{c \in \VPC(G),\mid\max(c)=k\}|,
\end{align*}
which finally implies that
\begin{align*}
P_{chr_S}(G)&=\sum_{c\in \VPC(G)} H_{\max(c)}.
\end{align*}
Observe that any valid coloring of $G$ with $\max(c)\leq N$ can be uniquely decomposed as $c=c'\circ c''$, where $c':V(G)\longrightarrow [n]$
is a valid packed coloring for a certain $n\leq N$, and $c'':[n]\longrightarrow [N]$ is a strictly increasing map. Therefore,
\begin{align*}
|\{\mbox{valid coloring of $G$ of maximum $\leq N$}\}|&=\sum_{c\in \VPC(G)} \binom{N}{\max(c)}\\
&=\sum_{c\in \VPC(G)} H_{\max(c)}(N)\\
&=P_{chr_S}(G)(N). \qedhere
\end{align*}
\end{proof}

\begin{remark}
If $V$ is a  non necessarily unitary, commutative bialgebra, the unique double bialgebra morphism from
$(\calF_V[\bfG],m,\Delta,\delta)$ to $(\K[X],m,\Delta,\delta)$ is $P_{chr_S}\circ \Theta_V$ (which is indeed a double
bialgebra morphism by composition). It sends any $V$-linearly decorated mixed graph $G$ to 
\[P_{chr_S}\circ \Theta_V(G)=\left(\prod_{x\in V(G)} \epsilon_V\circ d_G(x)\right) P_{chr_S}(\overline{G}),\]
where $d_G$ is the decoration map of $G$ and $\overline{G}$ the underlying mixed graph.
\end{remark}

Let us now recover the weak chromatic polynomial of \cite{Beck2012}.

\begin{defi}
Let $G$ be a mixed graph.
\begin{enumerate}
\item A weak valid coloring of $G$ is a map $c:V(G)\longrightarrow \N_{>0}$ such that
\begin{align*}
&\forall x,y\in V(G),& x\arc{G} y&\Longrightarrow c(x)\leq c(y),\\
&&x\arete{G} y&\Longrightarrow c(x)\neq c(y).
\end{align*}
\item A weak valid coloring $c$ of $G$ is packed if $c(V(G))=[\max(c)]$.
The set of weak valid packed colorings of $G$ is denoted by $\WVPC(G)$.
\end{enumerate}
\end{defi}

We are going to use the action $\leftsquigarrow$ of the monoid of characters on the set of morphisms and the map $\theta$ given in (\ref{eqtheta}). 

\begin{notation}
Let $\lambda_W:\calF[G]\longrightarrow \K$ defined on any mixed graph $G$  by
\begin{align}
\label{eqlambdaW}
\lambda_W(G)&=\begin{cases}
1\mbox{ if }E(G)=\emptyset,\\
0\mbox{ otherwise}.
\end{cases}
\end{align}
This is obviously a character.
\end{notation}

\begin{cor}\label{cor3.4}
We consider $P_{chr_W}=\theta(\lambda_W)=P_{chr_S}\leftsquigarrow \lambda_W$. Then
$P_{chr_W}=P_{chr_S}\leftsquigarrow \lambda_W:(\calF[\bfG],m,\Delta)\longrightarrow(\K[X],m,\Delta)$ is a Hopf algebra morphism. 
If $N\in \N$, $P_{chr_W}(G)(N)$ is the number of weak valid colorings $c$ such that $\max(c)\leq N$: we recover the weak chromatic polynomial of \cite{Beck2012}. 
\end{cor}

\begin{proof}
As in the proof of Proposition \ref{prop3.2}, we obtain that for any mixed graph $G$,
\begin{align*}
P_{chr_W}(G)&=\sum_{c \in \WVPC(G)}H_{\max(c)}.
\end{align*}
Observe that any weak valid coloring of $G$ with $\max(c)\leq N$ can be uniquely decomposed as $c=c'\circ c''$, where $c':V(G)\longrightarrow [n]$
is a weak valid packed coloring for a certain $n\leq N$, and $c'':[n]\longrightarrow [N]$ is a strictly increasing map. Therefore,
\begin{align*}
|\{\mbox{weak valid coloring of $G$ of maximum $\leq N$}\}|&=\sum_{c\in \WVPC(G)} \binom{N}{\max(c)}\\
&=\sum_{c\in \WVPC(G)} H_{\max(c)}(N)\\
&=P_{chr_W}(G)(N). \qedhere
\end{align*}
\end{proof}

\begin{remark}
Let $G\in \calG[X]$. We denote by $\eq^c_W[G]$ the set of equivalences $\sim\in \eq^c[G]$ such that
\begin{align*}
&\forall x,y\in V(G),&x\arc{G} y&\Longrightarrow x\nsim y.
\end{align*}
Then, for any mixed graph $G$,
\[P_{chr_W}(G)=\sum_{\sim\in \eq^c[G]} \lambda_W(G\mid \sim)P_{chr_S}(G/\sim) =
\sum_{\sim\in \eq^c_W[G]} P_{chr_S}(G/\sim).\]
\end{remark}

\begin{example}\label{ex3.1}
For $n\geqslant 3$, let $G_n$  be the following mixed graph:
\begin{align*}
V(G_n)&=[n],&
E(G_n)&=\{\{1,n\}\},&
A(G_n)&=\{(i,i+1)\mid i\in [n-1]\}.
\end{align*}
In other terms,
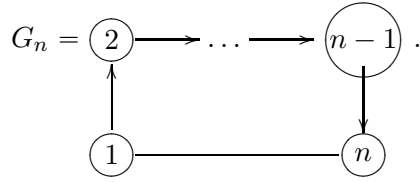
\[G_n=\xymatrix{\rond{2}\ar[r]&\ldots\ar[r]&\rond{n-1}\ar[d]\\
\rond{1}\ar[u]\ar@{-}[rr]&&\rond{n}}.\]
Weak valid colorings of $G_n$ are non-decreasing maps $c:[n]\longrightarrow \N_{>0}$, such that $c(1)\neq c(n)$. Therefore,
\[P_{chr_W}(G_n)=\frac{X(X+1)\ldots(X+n-1)}{n!}-X.\]
Valid colorings of $G_n$ are strictly increasing maps $c:[n]\longrightarrow \N_{>0}$. Therefore,
\[P_{chr_S}(G_n)=\frac{X(X-1)\ldots (X-n+1)}{n!}.\]
\end{example}

With the help of \cite[Propositions 3.10 and 5.2]{Foissy40}, we now define a homogeneous morphism 
$P_0:\calF[\bfG]\longrightarrow \K[X]$ with the help of the element $\mu\in \calF[\bfG]_1^*$ defined by
\[\mu(\grun)=1,\]
where $\grun$ is the unique (up to an isomorphism) mixed graph with only one vertex.
Then, if $G$ is a mixed graph,
\begin{align*}
\mu^{\otimes k}\circ \tdelta^{(k-1)}(G)&=\begin{cases}
|L_k(G)| \mbox{ if }k=|V(G)|,\\
0\mbox{ otherwise}.
\end{cases}
\end{align*}
We denote by $\ell(G)$ the cardinality of $L_n(G)$, that is to say the number of bijections 
$c:V(G)\longrightarrow [n]$ such that
\begin{align*}
&\forall x,y\in V(G),&x\arc{G} y&\Longrightarrow c(x)<c(y),
\end{align*}
and finally:

\begin{cor}\label{cor3.5}
For any mixed graph $G$, we put
\begin{align*}
\lambda_0(G)&=\frac{\ell(G)}{|V(G)|!},&
P_0(G)&=\lambda_0(G)X^{|V(G)|}.
\end{align*}
Then $\lambda_0$ is a character of $\calF[\bfG]$ and $P_0:(\calF[\bfG],m,\Delta)\longrightarrow (\K[X],m,\Delta)$ 
is a bialgebra morphism.
\end{cor}

For any graph $G$, $P_0(G)(1)=\lambda_0(G)$. From \cite[Corollary 3.11]{Foissy40},
$P_0=P_{chr_S}\leftsquigarrow \lambda_0$. Therefore:

\begin{cor}
For any  mixed graph $G$ with $n$ vertices,
\begin{align*}
\ell(G)X^n&=\sum_{\sim\in\eq^c[G]}\ell(G\mid \sim)P _{chr_S}(G/\sim).
\end{align*}
\end{cor}

\subsection{Invertible characters}

Let us fix a non unitary, commutative and cocommutative bialgebra $(V,\cdot,\delta_V)$. 
The product of the dual algebra is denoted by $\star_V$; its unit is the counit $\epsilon_V$. 
Let us now study the monoid of characters of $(\calF_V[\bfG],m,\delta)$, whose product is denoted by $\star$,
and in particular let us look for the group of its invertible elements. We shall use the following lemma:

\begin{lemma}\label{lemma3.9}
Let $(B,m,\delta)$ be a graded bialgebra. In particular, its homogeneous component of degree $0$ is a subbialgebra.
Let $\lambda$ be a character of $B$. Then $\lambda$ is an invertible character of $B$ if, and only if, 
its restriction $\lambda_0$ to $B_0$ is invertible in the algebra $B_0^*$. 
In the particular case where $B_0$ is generated by a family $(x_i)_{i\in I}$ of group-like elements,
$\lambda$ is an invertible character if, and only if, $\lambda(x_i)\neq 0$ for any $i\in I$. 
\end{lemma}

\begin{proof}
We shall denote by $\pi_k$ the canonical projection on $B_k$ for any $k\in \N$. We put
\begin{align*}
\rho_L&=(\pi_0\otimes \id)\circ \delta,&\rho_R&=(\pi\otimes \pi_0)\circ \delta.
\end{align*}
As $\pi_0:B\longrightarrow B_0$ is a bialgebra map, $(B,\rho_L,\rho_R)$ is a $B_0$-bicomodule.
For any $x\in B_n$ with $n\geq 1$, we put
\[\delta'(x)=\delta(x)-\rho_L(x)-\rho_R(x).\]
By homogeneity of $\delta$,
\[\delta(x)=\sum_{i=0}^n (\pi_i\otimes \pi_{n-i})\circ \delta(x)=\rho_L(x)+
\underbrace{\sum_{i=1}^{n-1}  (\pi_i\otimes \pi_{n-i})\circ \delta(x)}_{\delta'(x)}+\rho_R(x).\]

$\Longrightarrow$. Let us denote by $\mu$ the inverse of $\lambda$ in the monoid of characters of $B$.
We put $\mu_0=\mu_{\mid B_0}$. For any $x\in B_0$,
\[\lambda_0*\mu_0(x)=(\lambda \otimes \mu)\circ \delta_0(x)=
(\lambda \otimes \mu)\circ \delta(x)=\lambda*\mu(x)=\epsilon(x).\]
Similarly, $\mu_0*\lambda_0(x)=\epsilon(x)$, so $\lambda_0$ is invertible in $B_0^*$.\\

$\Longleftarrow$. Let us define $\mu_n:B_n\longrightarrow \K$ for any $n$, such that if $x\in B_n$,
\[\left(\lambda \otimes \sum_{i=0}^n\mu_i\right)\circ \delta(x)=\epsilon(x).\]
We proceed by induction on $n$. If $n=0$, we us take $\mu_0$ the inverse of $\lambda_0$ in $B_0^*$.
Let us assume $n\geq 1$ and $\mu_0,\ldots,\mu_{n-1}$ defined. We first define $\nu:B_n\longrightarrow \K$ by
\begin{align*}
&\forall x\in B_n,&\nu(x)=\epsilon(x)-\left(\lambda \otimes \sum_{i=0}^{n-1}\mu_i\right)
(\delta'(x)+\rho_R(x)).
\end{align*}
This is well-defined, as $\displaystyle \delta'(x)+\rho(x)\in B\otimes \bigoplus_{i=0}^{n-1}B_i.$.
We then put $\mu_n=(\mu_0\otimes \nu)\circ \rho_L:B_n\longrightarrow \K$. 
As $\rho_L=(\pi_0\otimes \id)\circ \delta$, by homogeneity of $\delta$, $\rho_L(B_n)\subseteq B_0\otimes B_n$
and $\mu_n$ is well-defined. For any $x\in B_n$,
\begin{align*}
\left(\lambda \otimes \sum_{i=0}^n \mu_n\right)\circ \delta(x)&=(\lambda \otimes \mu_n)\circ \rho_L(x)
+\epsilon(x)-\nu(x)\\
&=(\lambda \otimes \mu \circ \nu)\circ (\id \otimes \rho_L)\circ \rho_L(x)+\epsilon(x)-\nu(x)\\
&=(\lambda \otimes \mu \circ \nu)\circ (\delta_0 \otimes \id)\circ \rho_L(x)+\epsilon(x)-\nu(x)\\
&=((\lambda_0*\mu_0)\circ \nu)\circ \rho_L(x)+\epsilon(x)-\nu(x)\\
&=(\epsilon_0\circ \nu)\circ \rho_L(x)+\epsilon(x)-\mu(x)\\
&=\nu(x)+\epsilon(x)-\mu(x)\\
&=\epsilon(x).
\end{align*}
Considering $\displaystyle \mu=\sum_{i=0}^\infty \mu_i\in B^*$, by construction $\lambda*\mu=\epsilon$.
Similarly, we can  define $\mu'\in B^*$ such that $\mu'*\lambda=\epsilon$. Then, as the convolution product $*$
is associative, $\mu'=\mu$ and $\lambda$ is invertible in $B^*$. Let us now prove that $\mu$ is a character.
We work in the algebra $(B\otimes B)^*$, whose convolution product is also denoted by $*$.
For any $x,y\in B$, with Sweedler's notation $\displaystyle \delta(z)=\sum_z z^{(1)}\otimes z^{(2)}$ for any $z\in B$,
\begin{align*}
(\mu\circ m)*(\lambda \circ m)(x\otimes y)&=\sum_x\sum_y \mu\left(x^{(1)}y^{(1)}\right)\lambda\left(x^{(1)}y^{(1)}\right)\\
&=\sum_{xy} \mu((xy)^{(1)})\lambda ((xy)^{(2)})\\
&=\epsilon(xy)\\
&=\epsilon(x)\epsilon(y)\\
&=\epsilon_{B\otimes B}(x\otimes y). 
\end{align*}
Similarly, $(\lambda \circ m)*(\mu\circ m)=\epsilon_{B\otimes B}$, so, as $\lambda$ is a character,
\[\mu\circ m=(\lambda \circ m)^{*-1}=(\lambda \otimes \lambda)^{*-1}=\mu\otimes \mu.\]
So $\mu$ is indeed a character of $B$.\\

Let us now consider the particular case where $B_0$ is generated by a family $(x_i)_{i\in I}$ of group-like elements.

$\Longrightarrow$.  If $\lambda$ is an invertible character, denoting its inverse by $\nu$, for any $i\in I$,
\[\lambda*\nu(x_i)=\epsilon(x_i)=1=\lambda(x_i)\nu(x_i),\]
so $\lambda(x_i)\neq 0$. 

$\Longleftarrow$. Let us assume that $\lambda$ is a character of $B$ such that $\lambda(x_i)\neq 0$ for any $i\in I$. 
In order to prove that $\lambda$ is an invertible character, it is enough to prove that $\lambda_0$ is invertible in $B_0^*$.
By hypothesis, $B_0$ has a basis $(y_j)_{j\in J}$ of noncommutative monomials in $(x_i)_{i\in I}$.
By multiplicativity, for any $j\in J$, $y_j$ is group-like and $\lambda(y_j)\neq 0$. We then define $\mu_0\in B_0^*$ by
\begin{align*}
&\forall j\in J,&\mu(y_j)&=\frac{1}{\lambda(y_j)}.
\end{align*}
Then for any $j\in J$,
\[\lambda_0*\mu_0(y_j)=\mu_0*\lambda_0(y_j)=\lambda_0(y_j)\mu_0(y_j)=1=\epsilon(y_j),\]
so $\lambda_0$ is invertible in $B_0^*$.  \end{proof}

In order to use this lemma, let us introduce a grading of $(\calF_V[\bfG],m,\delta)$.

\begin{prop} \label{prop4.8}
For any $V$-linearly decorated mixed graph $G$, we denote by $\cc(G)$ the number of connected components of $G$ and we put
\[\deg(G)=|V(G)|-\cc(G).\]
This defines a grading of  the bialgebra $(\calF_V[\bfG],m,\delta)$.
\end{prop}

\begin{proof}
Note that for any graph $G$, $\deg(G)\geqslant 0$. Let $G$ and $H$ be two $V$-linearly decorated mixed graphs. Then
\begin{align*}
|V(GH)|&=|V(G)|+|V(H)|,&\cc(GH)&=\cc(G)+\cc(H),
\end{align*}
so $\deg(GH)=\deg(G)+\deg(H)$. Let $G$ be a $V$-linearly decorated mixed graph and $\sim\in \eq^c[G]$. 
We denote by $\cl(\sim)$ the number of equivalence classes of $\sim$. As $\sim\in \eq^c[G]$,
\begin{align*}
|V(G\mid \sim)|&=|V(G)|,&\cc(G\mid \sim)&=\cl(\sim).
\end{align*}
Moreover, the connected components of $G/\sim$ are the contractions of the connected components of $G$, so
\begin{align*}
|V(G/\sim)|&=\cl(\sim),&\cc(G/\sim)&=\cc(G).
\end{align*}
We obtain that $\deg(G/\sim)+\deg(G\mid \sim)=|V(G)|-\cl(\sim)+\cl(\sim)-\cc(G)=\deg(G)$. So $(\calF_V[\bfG],m,\delta)$ is graded.
\end{proof}

For any graph $G$, $\deg(G)=0$ if, and only if, $E(G)=A(G)=\emptyset$. The subbialgebra $\calF_V[\bfG]_{\deg=0}$
of elements of degree $0$ is the symmetric algebra generated by elements $\xymatrix{\rond{v}}$, with $v\in V$.
The coproduct of such an element is given by the coproduct of $V$,
\[\delta(\xymatrix{\rond{v}})=\xymatrix{\rond{v'}}\otimes \xymatrix{\rond{v''}}.\]

\begin{prop}
Let $\lambda \in \Char(\calF_V[\bfG])$. We define a map $\lambda_V\in V^*$ by
\begin{align*}
&\forall v\in V,&\lambda_V(v)=\lambda(\xymatrix{\rond{v}}).
\end{align*}
Then $\lambda$ is invertible in $(\Char(\calF_V[\bfG]),\star)$ if, and only if, $\lambda_V$ is invertible in $(V^*,\star_V)$.
\end{prop}

\begin{proof}
$\Longrightarrow$. Let us assume that $\lambda$ is an invertible character. Denoting by $\mu$ its inverse,
$\mu_{\mid V}$ provides an inverse of $\lambda_V$ in $V^*$. \\

$\Longleftarrow$. Let us assume that $\lambda_V$ is invertible in $V^*$. 
By Lemma \ref{lemma3.9}, it is enough to prove that $\lambda_0$ is invertible in the algebra $\calF_V[\bfG]_0$. 
By construction of the graduation, $\calF_V[\bfG]_0$ is the symmetric algebra generated by $V$. 
Extending multiplicatively the inverse of $\lambda_V$ to $\calF_V[\bfG]_0$, we obtain an inverse of $\lambda_0$. 
\end{proof}

In the particular case where $V=\K$:

\begin{cor}
Let $\lambda$ be a character of $\calF[\bfG]$. It  is invertible in the monoid $(\Char(\calF[\bfG]),\star)$ if, and only if,
\[\lambda(\grun)\neq 0.\]
\end{cor}

\begin{proof}
This is implied by Lemma \ref{lemma3.9}, with the family of group-like elements reduced to $\grun$. 
\end{proof}

In particular, $\lambda_W$ and $\lambda_0$ are invertible. Their inverses are denoted respectively by $\nu_W$ and $\mu_S$. 
We also put $\mu_W=\mu_S\star \lambda_W$. Then, as $P_{chr_W}=P_{chr_S}\leftsquigarrow \lambda_W$ 
and $P_0=P_{chr_S}\leftsquigarrow \lambda_0$, we obtain
\begin{align*}
P_{chr_S}&=P_{chr_W}\leftsquigarrow \nu_W,&
P_{chr_S}&=P_0\leftsquigarrow \mu_S,&
P_{chr_W}&=P_0\leftsquigarrow \mu_W.
\end{align*}

\begin{prop}\label{prop3.13}
For any mixed graph $G$,
\begin{align*}
P_{chr_S}(G)&=\sum_{\sim\in \eq^c[G]} \lambda_0(G/\sim) \mu_S(G\mid \sim) X^{\cl(\sim)}
=\sum_{\sim\in \eq^c[G]} \nu_W(G\mid \sim) P_{chr_W}(G/\sim),\\
P_{chr_W}(G)&=\sum_{\sim\in \eq^c[G]} \lambda_0(G/\sim) \mu_W(G\mid \sim) X^{\cl(\sim)},
\end{align*}
where $\cl(\sim)$ is the number of classes of $\sim$. 
\end{prop}

\begin{cor}\label{cor3.14}
If $G$ is a connected mixed graph, then $\mu_S(G)$ is the coefficient of $X$ in $P_{chr_S}(G)$
whereas $\mu_W(G)$ is the coefficient of $X$ in $P_{chr_W}(G)$. Moreover, $\lambda_0(G)$ is the coefficient of $X^{|V(G)|}$
in both $P_{chr_S}(G)$ and $P_{chr_W}(G)$.
\end{cor}

\begin{proof}
As $G$ is connected, the unique element $\sim$ of $\eq^c[G]$ with $\cl(\sim)=1$ is  $\sim_L$, which has for only class $V(G)$. 
So the coefficient of $X$ in $P_{chr_S}(G)$ is, as $\lambda_0$ and $\epsilon$ coincide on $\calF_V[\bfG]_{\deg=0}$, equal to
\begin{align*}
\lambda_0(G/\sim_L) \mu_S(G\mid \sim_L)&=(\lambda_0\otimes \mu_S)\circ \rho_L(G)
=(\epsilon\otimes \mu_S)\circ \rho_L(G)=\mu_S(G).
\end{align*}
Similarly, the unique equivalence $\sim\in \eq^c[G]$ such that $\cl(\sim)=|V(G)|$ is $\sim_R$, which has for classes the connected components the singletons
(or in other words, $\sim_R$ is the equality of $V(G)$).
So the coefficient of $X^{|V(G)|}$ in $P_{chr_S}(G)$ is, as $\mu_S$ and $\epsilon$ coincide on $\calF_V[\bfG]_{\deg=0}$ equal to
\begin{align*}
\lambda_0(G/\sim_R) \mu_S(G\mid \sim_R)&=(\lambda_0\otimes \mu_S)\circ \rho_R(G)
=(\lambda_0\otimes \epsilon)\circ \rho_R(G)=\lambda_0(G).
\end{align*}
The proof is similar for the weak chromatic polynomial.
\end{proof}

\begin{remark}One can define an infinitesimal character $\mu'_S$ as follows: for any mixed graph $G$,
\[\mu'_S(G)=\begin{cases}
\mu_S(G)\mbox{ if $G$ is connected},\\
0\mbox{ otherwise}.
\end{cases}\]
This infinitesimal character is equal to $\ln(\epsilon_\delta)$ and is studied in \cite[Proposition 4.1]{Foissy40}.
It is closely related to the eulerian idempotent.
Consequently, as the coefficient of $X$ in $H_k(X)$ is $\dfrac{(-1)^{k+1}}{k}$ if $k\geq 1$, we obtain that for any connected mixed graph $G$,
\begin{align*}
\mu_S(G)&=\sum_{c\in \VPC(G)} \frac{(1)^{\max(c)+1}}{\max(c)}=\sum_{k=1}^\infty |\{c\in \VPC(G)\mid \max(c)=k\}| \frac{(-1)^{k+1}}{k}.
\end{align*}
When $G$ is a rooted tree, we recover Murua's coefficients \cite{Arizmendi2022,Murua2006}, which appear in the analysis of the
continuous Baker--Campbell--Hausdorff problem. 
\end{remark}

\begin{example} \label{ex3.2}
In order to improve the  readability, we shall write $\grdeuxoo$ if there are two arcs of opposite directions between two vertices. Examples of chromatic polynomials are given in Table \ref{table3}.
\end{example}

\begin{table}\label{table3}
\[\begin{array}{|c||c|c|c|c|c|c|}
\hline G&P_{chr_S}(G)&P_{chr_W}(G)&\lambda_0(G)&\mu_S(G)&\mu_W(G)&\nu_W(G)\\
\hline\hline \grdeux&X(X-1)&X(X-1)&1&-1&-1&0\\
\hline&&&&&&\\[-4mm]
\grdeuxo&\dfrac{X(X-1)}{2}&\dfrac{X(X+1)}{2}&\dfrac{1}{2}&-\dfrac{1}{2}&\dfrac{1}{2}&-1\\[2mm]
\hline \grdeuxoo&0&X&1&0&1&-1\\
\hline &&&&&&\\[-4mm]
 \grtrois&X(X-1)^2&X(X-1)^2&1&1&1&0\\
\hline&&&&&&\\[-4mm]
\grtroisoe &\dfrac{X(X-1)^2}{2}&\dfrac{X(X+1)(X-1)}{2}&\dfrac{1}{2}&\dfrac{1}{2}&-\dfrac{1}{2}&0\\[2mm]
\hline&&&&&&\\[-4mm]
\grtroiseo &\dfrac{X(X-1)^2}{2}&\dfrac{X(X+1)(X-1)}{2}&\dfrac{1}{2}&\dfrac{1}{2}&-\dfrac{1}{2}&0\\
\hline &&&&&&\\[-4mm]
\grtroisooe&0&X(X-1)&1&0&-1&0\\
\hline&&&&&&\\[-4mm]
 \grtroisoo&\dfrac{X(X-1)(X-2)}{6}&\dfrac{X(X+1)(X+2)}{6}&\dfrac{1}{6}&\dfrac{1}{3}&\dfrac{1}{3}&1\\[2mm]
 \hline&&&&&&\\[-4mm]
 \grtroisoodeux&\dfrac{X(2X-1)(X-1)}{6}&\dfrac{X(2X+1)(X+1)}{6}&\dfrac{1}{3}&\dfrac{1}{6}&\dfrac{1}{6}&1\\[2mm]
\hline&&&&&&\\[-4mm]
\grtroisooo&0&\dfrac{X(X+1)}{2}&\dfrac{1}{2}&0&\dfrac{1}{2}&1\\[2mm]
\hline&&&&&&\\[-4mm]\grtroisooodeux&0&\dfrac{X(X+1)}{2}&\dfrac{1}{2}&0&\dfrac{1}{2}&1\\
\hline &&&&&&\\[-4mm]
\grtroisoooo&0&X&1&0&1&1\\
\hline &&&&&&\\[-4mm]
\grcompletun&X(X-1)(X-2)&X(X-1)(X-2)&1&2&2&0\\
\hline &&&&&&\\[-4mm]
\grcompletdeux&\dfrac{X(X-1)(X-2)}{2}&\dfrac{X^2(X-1)}{2}&\dfrac{1}{2}&1&0&0\\
\hline &&&&&&\\[-4mm]
\grcomplettrois&\dfrac{X(X-1)(X-2)}{3}&\dfrac{X(X-1)(X+1)}{3}&\dfrac{1}{3}&\dfrac{2}{3}&-\dfrac{1}{3}&2\\
\hline &&&&&&\\[-4mm]
\grcompletquatre&\dfrac{X(X-1)(X-2)}{6}&\dfrac{X(X-1)(X+4)}{6}&\dfrac{1}{6}&\dfrac{1}{3}&-\dfrac{2}{3}&2\\
\hline &&&&&&\\[-4mm]
\grcompletcinq&0&X&1&0&1&3\\
\hline &&&&&&\\[-4mm]
\grcompletsix&\dfrac{X(X-1)(X-2)}{6}&\dfrac{X(X+1)(X+2)}{6}&\dfrac{1}{6}&\dfrac{1}{3}&\dfrac{1}{3}&3\\
\hline \end{array}\]
\caption{Examples of chromatic polynomials}
\end{table}

\begin{example}
Let us consider again the graph $G_n$ of Example \ref{ex3.1}. We obtain
\begin{align*}
\mu_S(G_n)&=\frac{(-1)^n}{n!},& \mu_W(G_n)&=\frac{(n-1)!}{n!}-1=-\frac{n-1}{n},&
\lambda_0(G_n)&=\frac{1}{n!}.
\end{align*}
\end{example}

Here is an example of application. Recall that $\varpi_0:\bfG\longrightarrow \bfG_{ac}$, defined in (\ref{eqvarpi0}), is the projection vanishing on non acyclic graphs.

\begin{prop}
\label{prop3.15}
\begin{enumerate}
\item $P_{chr_S}\circ \varpi_0=P_{chr_S}$.
\item Let $G$ be an acyclic mixed graph. Then $P_{chr_S}(G)$ is of degree $|V(G)|$ and its leading term is $\ell(G)$. 
\end{enumerate}
\end{prop}

\begin{proof}
1. Note that $P_{chr_S}\circ \varpi_0$ and $P_{chr_S}$ are bialgebra morphisms from $\calF[\bfG]$ to $\K[X]$. Moreover,
\begin{align*}
\epsilon_\delta\circ P_{chr_S}\circ \varpi_0&=\epsilon_\delta\circ \varpi_0=\epsilon_\delta=\epsilon_\delta\circ P_{chr_S}.
\end{align*}
By unicity,  $P_{chr_S}\circ \varpi_0=P_{chr_S}$. \\

2. Let $G$ be an acyclic graph. As $P_{chr_S}=P_0\leftsquigarrow \mu_S$,
\[P_{chr_S}(G)=\sum_{\sim\in \eq_c[G]} \mu_S (G\mid \sim)\ell(G/\sim)X^{\cl(\sim)},\]
which implies that $\deg(P_{chr_S}(G))\leqslant |V(G)|$. If $\sim$ is the equality of $V(G)$,
then $G\mid \sim$ is a graph with no edge, so $\mu_S(G\mid \sim)=1$. We obtain that
\[P_{chr_S}(G)=\ell(G)X^{|V(G)|}+\mbox{terms of degree }<|V(G)|.\]
As $G$ has no cycle, $\ell(G)\neq 0$, so $\deg(P_{chr_S}(G))=|V(G)|$. 
\end{proof}

\section{From mixed graphs to acyclic oriented graphs}

\subsection{A double bialgebra epimorphism}

\begin{defi}
Let $G$ a mixed graph. An \emph{orientation} of $G$ is an oriented graph $H$ with $V(G)=V(H)$ and 
$A(H)=A(G)\sqcup E'$, where $E'$ is a set of arcs in bijection with $E(G)$, through a bijection respecting the extremities.
Such an orientation of $G$ is \emph{acyclic} if $H$ is an acyclic oriented graph. We denote by $O_{ac}(G)$ the set of acyclic orientations of $G$. 
\end{defi}

\begin{theo}\label{theo4.2}
Let $G$ be a mixed graph. We put
\[\Theta(G)=\sum_{G'\in \calO_{ac}(G)}G',\]
with the usual convention that this sum is $0$ if $\calO_{ac}(G)$ is empty. Then $\Theta$ is a double twisted bialgebra morphism
from $\bfG$ to $\bfG_{aco}$. 
\end{theo}

\begin{proof}
Firstly, $\Theta$ is indeed a species morphism from $\bfG$ to $\bfG_{aco}$. 
Let $G$ and $H$ be two mixed graphs, $G'$ and $H'$ be orientations of $G$ and $H$.
Then $G'H'$ is an acyclic orientation of $GH$ if, and only if, $G'$ and $H'$ are acyclic orientations of $G$ and $H$. 
This implies directly that $\Theta(GH)=\Theta(G)\Theta(H)$. Therefore, $\Theta$ is a twisted algebra morphism. \\

Let $G\in \bfG[I\sqcup J]$ be a mixed graph. 
If $J$ is not an ideal of $G$, then $\Delta_{I,J}(G)=0$. Moreover, for any orientation $G'$ of $G$,
$J$ is not an ideal of $G'$, so $\Delta_{I,J}(G')=0$. Hence, in this case,
\[(\Theta\otimes \Theta)\circ \Delta_{I,J}(G)=\Delta_{I,J}\circ \Theta(G)=0.\]
We now assume that $J$ is an ideal of $G$. Then
\begin{align*}
(\Theta\otimes \Theta)\circ \Delta_{I,J}(G)&=\sum_{(G',G'')\in \calO_{ac}(G_{\mid I})\times \calO_{ac}(G_{\mid J})}
G'\otimes G'',\\
\Delta_{I,J}\circ \Theta(G)&=\sum_{\substack{H\in \calO_{ac}(G),\\ \mbox{\scriptsize $J$ ideal of $H$}}}
H_{\mid I}\otimes H_{\mid J}. 
\end{align*}
We put
\begin{align*}
A&=\calO_{ac}(G_{\mid I})\times \calO_{ac}(G_{\mid J}),&
B&=\{H\in \calO_{ac}(G)\mid \mbox{$J$ ideal of $H$}\},
\end{align*}
and we consider the map
\begin{align*}
\psi&:\left\{\begin{array}{rcl}
B&\longrightarrow&A\\
H&\longrightarrow&(H_{\mid I},H_{\mid J}).
\end{array}\right.
\end{align*}
This is obviously well-defined. We now consider the map $\psi':A\longrightarrow B$,
sending any pair $(G',G'')$ to an orientation $\psi(G',G'')=H$ of $G$ defined in this way:
for any edge  $\{x,y\}$ of $G$,
\begin{itemize}
\item If $x,y\in I$, then orient this edge as in $G'$.
\item If $x,y\in J$, then orient this edge as in $G''$.
\item If $x\in I$ and $y\in J$, then orient this edge from $x$ to $y$.
\item  If $x\in J$ and $y\in I$, then orient this edge from $y$ to $x$.
\end{itemize}
As there is no arc in $H$ from $J$ to $I$, $J$ is an ideal of $H$. Moreover, as $G'$ and $G''$ are acyclic,
$H$ is acyclic: $\psi'$ is well-defined. We immediately obtain that $\psi'\circ \psi=\id_B$ 
and $\psi\circ \psi'=\id_A$, so $\psi$ and $\psi'$ are bijections. Therefore,
\begin{align*}
(\Theta\otimes \Theta)\circ \Delta_{I,J}(G)&=\sum_{(G',G'))\in A}G'\otimes G''
=\sum_{H\in B}H_{\mid I}\otimes H_{\mid J}=\Delta_{I,J}\circ \Theta(G).
\end{align*}
So $\Theta:(\bfG,m,\Delta)\longrightarrow(\bfG_{aco},m,\Delta)$ is a twisted bialgebra morphism.\\

Let $G\in \bfG[I\sqcup J]$ be a mixed graph and $\sim \in \eq(G)$. If $H$ is an orientation of $G$,
as the paths in $G$ and $H$ are the same, $\sim\in \eq_\sim(G)$ if, and only if, $\sim\in \eq_\sim(H)$. Hence,
if $\sim\notin \eq_c(G)$,
\[\delta_\sim\circ \Theta(G)=(\Theta\otimes\Theta)\circ \delta_\sim(G)=0.\]
Let us now assume that $\sim\in \eq_c(G)$. Then
\begin{align*}
\delta_\sim\circ \Theta(G)&=\sum_{\substack{H\in \calO_{ac}(G),\\ \mbox{\scriptsize $H/\sim$ acyclic}}}
H/\sim\otimes H\mid \sim,\\
(\Theta\otimes \Theta)\circ \delta_\sim(G)&=\sum_{(G',G'')\in \calO_{ac}(G/\sim)\times \calO_{ac}(G\mid\sim)}
G'\otimes G''.
\end{align*}
We put
\begin{align*}
C&=\calO_{ac}(G/\sim)\times \calO_{ac}(G\mid\sim),&D&=\{H\in \calO_{ac}(G)\mid  \mbox{$H/\sim$ acyclic}\}.
\end{align*}
If $H\in D$, then $H/\sim$ is an acyclic orientation of $G/\sim$ by definition of $D$
and $H\mid \sim$ is an acyclic orientation of $G\mid\sim$ by restriction. This defines a map
\[\phi:\left\{\begin{array}{rcl}
C&\longrightarrow&D\\
H&\longrightarrow&(H/\sim,H\mid \sim).
\end{array}\right.\]
Let us now consider $(G',G'')\in C$. We define an orientation of $G$ in the following way:
if $\{x,y\}$ is an edge of $G$,
\begin{itemize}
\item If $x\sim y$, then $\{x,y\}$ is an edge of $G\mid \sim$: orient it as in $G''$.
\item Otherwise, $\{\overline{x},\overline{y}\}$ is an edge or an arc of $G/\sim$:
orient $\{x,y\}$ as $\{\overline{x},\overline{y}\}$ in $G'$: as $G'$ is acyclic, this is unambiguous. 
\end{itemize}
Note that $H/\sim=G'$ and $H\mid \sim=G''$ by construction. Moreover, this is an acyclic orientation of $G$: 
if $x_1\rightarrow\ldots \rightarrow x_k\rightarrow x_1$ is a cycle in $H$,
as $G'$ is acyclic, necessarily $x_1\sim\ldots \sim x_k$, so this is a cycle in $G''$: as $G''$ is acyclic, this is not possible.
Moreover, $H/\sim=G'$ is acyclic, so this defines a map $\phi':D\longrightarrow C$ such that $\phi\circ \phi'=\id_D$. 

Let $H\in C$. We put $H'=\phi'\circ \phi(H)$. Let $(x,y)$ be an arc of $H$. If $x\sim y$,
then $(x,y)$ is an arc of $H\mid \sim$, so is an arc of $H'$.
Otherwise, $(\overline{x},\overline{y})$ is an arc of $H/\sim=H'/\sim$. 
If $(y,x)$ is an arc of $H'$, then $\overline{x}\rightarrow\overline{y}\rightarrow \overline{x}$ is a cycle in $H'/\sim$,
so $H/\sim$ is not acyclic: this is a contradiction. So $(x,y)$ is an arc of $H'$. Therefore, $H$ and $H'$ have the same arcs,
so are equal. We proved that $\phi'\circ \phi=\id_C$, so $\phi$ is a bijection. We obtain
\begin{align*}
\delta_\sim\circ \Theta(G)&=\sum_{H\in D}H/\sim\otimes H\mid \sim
=\sum_{(G',G'')\in C}G'\otimes G''=(\Theta\otimes \Theta)\circ \delta_\sim(G).
\end{align*}
So $\Theta$ is compatible with $\delta$. It is obviously compatible with the unit and both counits $\varepsilon_\Delta$ and $\epsilon_\delta$.
\end{proof}

\begin{cor} \label{cor4.3}
For any mixed graph $G$,
\[P_{chr_S}(G)=\sum_{H\in \calO_{ac}(G)}P_{chr_S}(H).\]
\end{cor}

\begin{proof}

By functoriality, $\calF[\Theta]:\calF[\bfG]\longrightarrow\calF[\bfG_{aco}]$ is a double bialgebra morphism.
By composition, $P_{chr_S}\circ \calF[\Theta]:\calF[\bfG]\longrightarrow\K[X]$ is a double bialgebra morphism.
By unicity of such a morphism, 
\[P_{chr_S}\circ \calF[\Theta]=P_{chr_S}. \qedhere\]
\end{proof}

\subsection{Erhahrt polynomials for mixed graphs}

\begin{prop}\label{prop4.4}
Let $q\in \K$. The two following maps are characters of $\calF[\bfG]$:
\begin{align*}
\ehr^{(q)}_{str}&:\left\{\begin{array}{rcl}
\calF[\bfG]&\longrightarrow&\K\\
G&\longmapsto&\begin{cases}
q^{|V(G)|}\mbox{ if }A(G)=\emptyset,\\
0\mbox{ otherwise},
\end{cases}
\end{array}\right.&
\ehr^{(q)}&:\left\{\begin{array}{rcl}
\calF[\bfG]&\longrightarrow&\K\\
G&\longmapsto&q^{|V(G)|},
\end{array}\right.
\end{align*}
with the convention $q^0=1$ even if $q=0$. 
Let $\Ehr^{(q)}_{str}$ and by $\Ehr^{(q)}$ be Hopf algebra morphisms from 
$(\calF[\bfG],m,\Delta)$ to $(\K[X],m,\Delta)$ given by
\begin{align*}
\Ehr^{(q)}_{str}&=\theta\left(\ehr^{(q)}_{str}\right)=P_{chr_S}\leftsquigarrow \ehr^{(q)}_{str},\\
\Ehr^{(q)}&=\theta\left(\ehr^{(q)}\right)=P_{chr_S}\leftsquigarrow \ehr^{(q)}.
\end{align*}
where $\theta$ is defined in (\ref{eqtheta}).

\begin{align*}
\epsilon_\delta \circ \Ehr_{str}^{(q)}&=\ehr_{str}^{(q)},&
\epsilon_\delta \circ \Ehr^{(q)}&=\ehr^{(q)}.
\end{align*}
Then, for any $n\in \N$, for any mixed graph $G$,
\begin{align*}
\Ehr_{str}^{(q)}(G)(n)&=q^{|V(G)|}|\{f:V(G)\longrightarrow [n]\mid \forall x,y\in V(G),\: 
x\arc{G} y\Longrightarrow f(x)<f(y)\}|,\\
\Ehr^{(q)}(G)(n)&=q^{|V(G)|}|\{f:V(G)\longrightarrow [n]\mid \forall x,y\in V(G),\: 
x\arc{G} y\Longrightarrow f(x)\leqslant f(y)\}|.
\end{align*}
Moreover, $\Ehr^{(q)}_{str}\circ \varpi_0=\Ehr^{(q)}_{str}$ (recall that $\varpi_0$ is defined in (\ref{eqvarpi0})).
\end{prop}

From now, we shall write simply $\ehr_{str}$, $\ehr$, $\Ehr_{str}$ and $\Ehr$ for $\ehr_{str}^{(1)}$, 
$\ehr^{(1)}$, $\Ehr_{str}^{(1)}$ and $\Ehr^{(1)}$. 

\begin{proof}
The maps $\ehr^{(q)}$ and $\ehr_{str}^{(q)}$ are obviously characters of $\calF[\bfG]$. 
Let $G$ be a mixed graph. For $k\geq 1$, we denote by $S_k(G)$ the set of surjective  maps $f:V(G)\longrightarrow [k]$
such that for any $x,y\in V(G)$,
\[x\arc{G} y\Longrightarrow f(x)\leqslant f(y).\]
Then
\begin{align*}
\Ehr_{str}^{(q)}(G)&=\sum_{k=1}^\infty \sum_{f\in S_k(G)} \ehr_{str}^{(q)}(G_{\mid f^{-1}(1)})\ldots 
\ehr_{str}^{(q)}(G_{\mid f^{-1}(k)}) H_k(X)\\
&=\sum_{k=1}^\infty \sum_{f\in S_k(G)}q^{|f^{-1}(1)|}\ehr_{str}^{(1)}(G_{\mid f^{-1}(1)})\ldots 
q^{|f^{-1}(k)|}\ehr_{str}^{(1)}(G_{\mid f^{-1}(k)}) H_k(X)\\
&=q^{|V(G)|}\sum_{k=1}^\infty \sum_{f\in S_k(G)}\ehr_{str}^{(1)}(G_{\mid f^{-1}(1)})\ldots 
\ehr_{str}^{(1)}(G_{\mid f^{-1}(k)}) H_k(X)\\
&=q^{|V(G)|}\Ehr_{str}(G),
\end{align*}
Similarly, $\Ehr^{(q)}(G)=q^{|V(G)|}\Ehr(G)$. We now study $\Ehr_{str}$ and $\Ehr$. 
Let $G$ be a mixed graph.
\[\Ehr(G)=\sum_{k=1}^\infty \sum_{f\in S_k(G)} H_k(X),\]
which gives the announced result. For $k\geq 1$, we denote by $S'_k(G)$ the set of surjective  maps $f:V(G)\longrightarrow [k]$
such that for any $x,y\in V(G)$,
\[x\arc{G} y\Longrightarrow f(x)< f(y).\]
By definition of $\ehr_{str}$,
\[\Ehr_{str}(G)=\sum_{k=1}^\infty \sum_{f\in S_k(G)} \ehr_{str}(G_{\mid f^{-1}(1)})\ldots 
\ehr_{str}(G_{\mid f^{-1}(k)}) H_k(X)=\sum_{k=1}^\infty \sum_{f\in S_k'(G)} H_k(X),\]
which implies the announced result.\\

Let us prove that $\Ehr^{(q)}_{str}\circ \varpi_0=\Ehr^{(q)}_{str}$. For this, it is enough to prove that
$\epsilon_\delta\circ \Ehr^{(q)}_{str}\circ \varpi_0=\epsilon_\delta \circ \Ehr^{(q)}_{str}$, that is to say 
$\ehr_{str}^{(q)}\circ \varpi_0=\ehr_{str}^{(q)}$. Let $G$ be a graph. If $G$ is not acyclic, then 
$\ehr_{str}^{(q)}\circ \varpi_0(G)=0$. Moreover, necessarily $A(G)\neq \emptyset$, so $\ehr_{str}^{(q)}=0$.
Otherwise, $\varpi_0(G)=G$ and $\ehr_{str}^{(q)}\circ \varpi_0(G)=\ehr_{str}^{(q)}(G)$. 
\end{proof}

\begin{remark}
By \cite[Corollary 3.12]{Foissy40},
\begin{align*}
\Ehr_{str}&=P_{chr_S}\leftsquigarrow \ehr_{str},&\Ehr&=P_{chr_S}\leftsquigarrow \ehr.
\end{align*}
\end{remark}

\begin{remark}
Classically, Ehrhart polynomials are attached to integral polytopes: given an integral polytope $P$, its Ehrhart polynomial $\Ehr_P$, evaluated in $t\in \N$,
gives the number on integral points of the dilated polytope $tP$. The duality principle states that the number of integral points in the interior of $tP$ is $(-1)^{\dim(P)}\Ehr_P(-t)$. 
Given a mixed graph $G$, after an arbitrary indexation of its vertices by $[n]$, we associate to it an integral polytope
\[P_G=\{(x_1,\ldots,x_n)\in [0,1]^n\mid \forall i,j\in [n],\: i\arc{G} j\Longrightarrow x_i\leq x_j\}.\]
Then $\Ehr_{P_G}=\Ehr_G$. 
\end{remark}

\begin{notation}
Let $G$ be a mixed graph. We denote by $S(G)$ the set of vertices $y\in V(G)$ such that there exists no $e\in A(G)$
such that $y$ is the final vertex of $e$ (set of \emph{sources} of $G$) and by $W(G)$ the set of vertices $x\in V(G)$ 
such that there exists no $e\in A(G)$ such that $x$ is the initial vertex of $e$ (set of \emph{wells} of $G$).
\end{notation}

\begin{prop} \label{propconvo}
Let $q,q'\in \K$. For any mixed graph $G$, denoting by $*$ the convolution product associated to $\Delta$,
\begin{align*}
\ehr_{str}^{(q)}*\ehr^{(q')}(G)&=q'^{|V(G)\setminus S(G)|}(q+q')^{|S(G)|},\\
\ehr^{(q)}*\ehr_{str}^{(q')}(G)&=q'^{|V(G)\setminus W(G)|}(q+q')^{|W(G)|}.
\end{align*}
In particular, if $S(G)\neq \emptyset$ and $W(G)\neq \emptyset$,
\begin{align*}
\ehr_{str}^{(-q)}*\ehr^{(q)}(G)&=\ehr^{(q)}*\ehr_{str}^{(-q)}(G)=0=\varepsilon_\Delta(G).
\end{align*}
\end{prop}

\begin{proof}
Indeed,
\begin{align*}
\ehr_{str}^{(q)}*\ehr^{(q')}(G)&=\sum q^{|I_1|} q'^{|I_2|},
\end{align*}
where the sum is over all partitions $V(G)=I_1\sqcup I_2$ such that if $x\arc{G} y$ in $V(G)$, then $(x,y)\in (I_1\times I_2)\cup I_2^2$,
that is to say such that $I_1\subseteq S(G)$. Hence,
\begin{align*}
\ehr_{str}^{(q)}*\ehr^{(q')}(G)&=\sum_{I_1\subseteq S(G)} q^{|I_1|} q'^{|V(G)\setminus I_1|}\\
&=q'^{|V(G)\setminus S(G)|}\sum_{I_1\subseteq S(G)} q^{|I_1|} q'^{|S(G)\setminus I_1|}\\
&=q'^{|V(G)\setminus S(G)|}(q+q')^{|S(G)|}. 
\end{align*}
The proof is similar for $\ehr_{str}^{(q)}*\ehr^{(q')}(G)$, replacing sources by wells. 
\end{proof}

\begin{cor} \label{cor5.6}
The inverse (for the convolution product $*$) of the restriction of $\ehr^{(q)}$ to $\calF[\bfG_{ac}]$ is 
the restriction of $\ehr_{str}^{(-q)}$.
\end{cor}

\begin{proof}
Let $G$ be an acyclic mixed graph. Obviously, if $G=1$, then
\[\ehr_{str}^{(-q)}*\ehr^{(q)}(G)1=\varepsilon_\Delta(G).\] 
Otherwise, as $G$ is acyclic, then $S(G)\neq \emptyset$ and $W(G)\neq \emptyset$. We can conclude with Proposition \ref{propconvo}. 
\end{proof}

Let us now prove the duality principle for Ehrhart polynomials:

\begin{cor} \label{cor4.7}
Let $G$ be an acyclic mixed graph. Then 
\[\Ehr_{str}(G)(-X)=(-1)^{|V(G)|}\Ehr(G)(X).\]
\end{cor}

\begin{proof}
We denote by $S_{\bfG_{ac}}$ the antipode of $(\calF[\bfG_{ac}],m,\Delta)$ and by $S$ the antipode of $(\K[X],m,\Delta)$. In particular, for any $P\in \K[X]$, $S(P(X))=P(-X)$.
As $\Ehr:(\calF[\bfG_{ac}],m,\Delta)\longrightarrow(\K[X],m,\Delta)$ is a Hopf algebra morphism,
\[\Ehr_{str}(G)(-X)=S\circ \Ehr_{str}(G)=\Ehr_{str}\circ S_{\bfG_{ac}}(G). \]
Therefore, by Corollary \ref{cor5.6},
\begin{align*}
\Ehr_{str}(G)(-1)&=S\circ \Ehr_{str}(G)(1)\\
&=\Ehr_{str}\circ S_{\bfG_{ac}}(G)(1)\\
&=\ehr_{str}\circ S_{\bfG_{ac}}(G)\\
&=\ehr_{str}^{*-1}(G)\\
&=\ehr^{(-1)}(G).
\end{align*}
This implies that $\Ehr_{str}\circ S$ is the Hopf algebra morphism $P_{chr_S}\leftsquigarrow \ehr^{(-1)}$:
\begin{align*}
\Ehr_{str}\circ S(G)&=(P_{chr_S}\otimes \ehr^{(-1)})\circ \delta(G)\\
&=(-1)^{|V(G)|}(P_{chr_S}\otimes \ehr)\circ \delta(G)\\
&=(-1)^{|V(G)|} \Ehr(G)(X).\qedhere
\end{align*}
\end{proof}

\begin{cor}
Let $G$ be an acyclic oriented graph. Then 
\begin{align*}
\Ehr_{str}(G)&=P_{chr_S}(G),&P_{chr_S}(G)(-1)&=(-1)^{|V(G)|},\\
\Ehr(G)&=P_{chr_W}(G),&P_{chr_W}(G)(-1)&=(-1)^{|V(G)|}\epsilon_\delta(G).
\end{align*}
\end{cor}

\begin{proof}
We work in $\calF[\bfG]$. For all acyclic oriented graphs $H$,
\[\ehr_{str}(H)=\epsilon_\delta(H)=\begin{cases}
1\mbox{ if }A(H)=\emptyset,\\
0\mbox{ otherwise}.
\end{cases}\]
Hence, $\Ehr_{str}(G)=P_{chr_S}(G)$: the restriction of $\Ehr$ to $\calF[\bfG_{aco}]$ is $P_{chr_S}$. 
Moreover,
\begin{align*}
\Ehr_{str}(G)(-1)&=(-1)^{|V(G)|}\Ehr(G)(1)=(-1)^{|V(G)|}\ehr(G)=(-1)^{|V(G)|}. 
\end{align*} 
For any oriented acyclic graph $H$, $\ehr(H)=1=\lambda_W(H)$, so $\Ehr(G)=P_{chr_W}(G)$. 
Moreover, 
\begin{align*}
\Ehr(G)(-1)&=(-1)^{|V(G)|}\Ehr_{str}(G)(1)=(-1)^{|V(G)|}\ehr_{str}(G)=(-1)^{|V(G)|}\epsilon_\delta(G).\qedhere 
\end{align*} 
\end{proof}

From Corollary \ref{cor4.3}, we obtain another proof of the following result, obtained in a different way in \cite[Theorem 3]{Beck2012}:

\begin{cor} \label{cor4.9}
Let $G$ be a mixed graph. Then
\[P_{chr_S}(G)(-1)=(-1)^{|V(G)|}|\calO_{ac}(G)|.\]
\end{cor}

\begin{proof}
Indeed,
\[P_{chr_S}(G)(-1)=\sum_{H\in \calO_{ac}(G)}P_{chr_S}(H)(-1)=\sum_{H\in \calO_{ac}(G)}(-1)^{|V(H)|}
=(-1)^{|V(G)|}|\calO_{ac}(G)|. \qedhere \]
\end{proof}

\begin{remark}
We recover the classical result on the chromatic polynomial when this corollary is applied to graphs \cite{Stanley1973}.
\end{remark}

From \cite[Corollary 2.3]{Foissy40}:

\begin{cor}\label{cor4.10}
Denoting by $S$ the antipode of $(\calF[\bfG],m,\Delta)$, for any mixed graph $G$,
\[S(G)=\sum_{\sim\in \eq^c[G]} (-1)^{\cl(\sim)}|\calO_{ac}(G/\sim)|\:(G\mid \sim).\]
\end{cor}

\section{Applications to characters on mixed graphs}

\subsection{Weak chromatic polynomial}

\begin{prop}\label{prop5.1}
Let $G$ be a totally acyclic mixed graph. Then 
\[P_{chr_W}(G)(-1)=\begin{cases}
0\mbox{ if }A(G)\neq \emptyset,\\
(-1)^{|V(G)|}|\calO_{ac}(G)| \mbox{ otherwise}.
\end{cases}\]
\end{prop}

\begin{proof}
If $A(G)=\emptyset$, then $P_{chr_W}(G)=P_{chr_S}(G)$, and the result comes from Corollary \ref{cor4.9}. Let us assume
that $A(G)\neq \emptyset$. We proceed by induction on $|E(G)|$. If $E(G)=\emptyset$, then by definition,
$P_{chr_W}(G)=\Ehr(G)$. By the duality principle for Ehrhart polynomials (Corollary \ref{cor4.7}),
\[P_{chr_W}(G)(-1)=(-1)^{|V(G)|} \Ehr_{str}(G) (1)=0,\]
as $A(G)\neq \emptyset$. Let us assume the result for all acyclic graph $H$ such that $|E(G)|>|E(H)|$ and $A(H)\neq \emptyset$. 
Let $e$ be an edge of $G$. 
We denote respectively by $G/e$  and by $G\setminus e$ the mixed graph obtained from $G$ by contraction of the edge $e$
respectively by deleting the edge $e$. From \cite[Proposition 6]{Beck2012}, 
\[P_{chr_W}(G)(-1)=P_{chr_W}(G\setminus e)(-1)-P_{chr_W}(G/e)(-1).\]
Moreover,  $G\setminus e$ and $G/e$ are mixed graph with at least one arc and strictly less edges
than $G$. Moreover, as $G$ is totally acyclic, $G/e$ and $G\setminus e$ are acyclic: we deduce that $P_{chr_W}(G\setminus e)(-1)
=P_{chr_W}(G/e)(-1)=0$.  Hence, $P_{chr_W}(G)(-1)=0$. 
\end{proof}

\begin{remark}
If $G$ is not totally acyclic, no interpretation of $P_{chr_W}(G)(-1)$, and even of its sign, is known. For example, 
if $G_n$ is the graph of Example \ref{ex3.1}, then $P_{chr_W}(G_n)(-1)=1$ for any $n\geq 2$. 
\end{remark}

We recover the interpretation of \cite{Beck2015-2} of the values of the weak chromatic polynomial at negative values:

\begin{cor}\label{cor5.2}
Let $G$ be a totally acyclic mixed graph and $k\in \N$. Then $(-1)^{|V(G)|} P_{chr_W}(G)(-k)$ is the number of pairs $(H,f)$ such that:
\begin{itemize}
\item $H$ is an acyclic orientation of $G$.
\item $f$ is a $k$-coloring of $G$ compatible with $H$, that is,
\begin{align*}
&\forall x,y\in V(G),& \begin{cases}
x \arc{G} y&\Longrightarrow f(x)<f(y),\\
x \arc{H}y&\Longrightarrow f(x)\leq f(y).
\end{cases}
\end{align*}
\end{itemize}
\end{cor}

\begin{proof}
By compatibility of $P_{chr_W}$ with the coproduct $\Delta$,
\begin{align*}
P_{chr_W}(G)(-k)&=\sum_{\substack{f:V(G)\longrightarrow [k],\\ x\arc{G}y\Longrightarrow f(x)\leq f(y)}}
P_{chr_W}\left(G_{\mid f^{-1}(1)}\right)(-1)\ldots 
P_{chr_W}\left(G_{\mid f^{-1}(k)}\right)(-1).
\end{align*}
By the Proposition \ref{prop5.1}, if $G_{\mid f^{-1}(i)}$ has an arc, then 
$P_{chr_W}\left(G_{\mid f^{-1}(1)}\right)(-1)=0$. Therefore,
\begin{align*}
P_{chr_W}(G)(-k)&=\sum_{\substack{f:V(G)\longrightarrow [k],\\ x\arc{G}y\Longrightarrow f(x)< f(y)}}
P_{chr_W}\left(G_{\mid f^{-1}(1)}\right)(-1)\ldots P_{chr_W}\left(G_{\mid f^{-1}(k)}\right)(-1)\\
&=\sum_{\substack{f:V(G)\longrightarrow [k],\\ x\arc{G}y\Longrightarrow f(x)< f(y)}}
(-1)^{|f^{-1}(1)|+\ldots+|f^{-1}(k)|}\prod_{i=1}^k|\calO_{ac}\left(G_{\mid f^{-1}(i)}\right)|\\
&=(-1)^{|V(G)|}\sum_{\substack{f:V(G)\longrightarrow [k],\\ x\arc{G}y\Longrightarrow f(x)< f(y)}}
\prod_{i=1}^k |\calO_{ac}\left(G_{\mid f^{-1}(i)}\right)|.
\end{align*}
We consider
\begin{align*}
A_k&=\{(f,H_1,\ldots,H_k)\mid f:V(G)\longrightarrow [k],\: H_i\in \calO_{ac}\left(G_{\mid f^{-1}(i)}\right)\},\\
B_k&=\{(H,f)\mid H\in \calO_{ac}(G), \:f:V(G)\longrightarrow [k]\mbox{ compatible with }H\}.
\end{align*}
The map $\phi:B_k\longrightarrow A_k$ which send $(H,f)$ to $(f,H_{\mid f^{-1}(1)},\ldots, H_{\mid f^{-1}(k)})$ is well-defined.
If $\phi(H,f)=\phi(H',f')$, then $f=f'$. Moreover, if $\{x,y\} \in E(G)$:
\begin{itemize}
\item If $f(x)=f(y)=i$, then $\{x,y\}$ is oriented in the same way in $H_{\mid f^{-1}(i)}$ and in $H'_{\mid f^{-1}(i)}$,
as these oriented graphs are equal. So $\{x,y\}$ is oriented in the same way in $H$ and in $H'$.
\item If $f(x)<f(y)$, as $f$ is compatible with $H$ and in $H'$, then $(x,y)\in E(H)$ and $(x,y)\in E(H')$.
\item If $f(x)>f(y)$, as $f$ is compatible with $H$ and in $H'$, then $(y,x)\in E(H)$ and $(y,x)\in E(H')$.
\end{itemize}
Therefore, $H=H'$: $\phi$ is injective. Let $(f,H_1,\ldots,H_k)\in A_k$. We define an orientation of $G$ as follows:
if $\{x,y\}\in E(G)$,
\begin{itemize}
\item if $f(x)=f(y)=i$, we keep the orientation of this edge in $H_i$.
\item If $f(x)<f(y)$, we orient this edge in $(x,y)$ in $H$.
\item If $f(x)>f(y)$, we orient this edge in $(y,x)$ in $H$.
\end{itemize}
By construction, for any $i\in [k]$, $H_{\mid f^{-1}(i)}=H_i$. Moreover, $f$ is compatible with $H$, by construction.
Consequently, $f$ is constant on any cycle of $H$. As the oriented graphs $H_i$ are acyclic, $H$ is acyclic. 
We obtain that $(H,f)\in B_k$ and $\phi(H,f)=(f,H_1,\ldots,H_k)$. Finally,
\[P_{chr_W}(G)(-k)=(-1)^{|V(G)|}|A_k|=(-1)^{|V(G)|}|B_k|. \qedhere\]
\end{proof}

\subsection{The character $\nu_W$}

Recall that $\nu_W$ is the inverse of $\lambda_W$ (defined in (\ref{eqlambdaW})) for the convolution $\star$ associated to $\delta$, 
and that the notations $\cc(G)$ and $\deg(G)$ are introduced in Proposition \ref{prop4.8}.
\begin{prop}\label{prop5.3}\begin{enumerate}
\item For any simple  graph $G$ with at least one edge, $\nu_W(G)=0$. 
\item For any oriented graph $G$, $\nu_W(G)=(-1)^{|V(G)|-\cc(G)}=(-1)^{\deg(G)}$. 
\end{enumerate}
\end{prop}

\begin{proof}
1. We denote by $\lambda_W^s$ the restriction of $\lambda_W$
to $\calF[\bfG_s]$. For any simple graph $G$, $\lambda_W(G)=\epsilon_\delta(G)$, so
$\lambda_W^s={\epsilon_\delta}_{\mid \calF_V[\bfG_s]}$. 
As $\calF[\bfG_s]$ is a double subbialgebra of $\calF[\bfG]$,
\[(\lambda_W^s)^{\star -1}=({\epsilon_\delta}_{\mid \calF[\bfG_s]})^{\star -1}
=({\epsilon_\delta}^{\star -1})_{\mid \calF[\bfG_s]}={\epsilon_\delta}_{\mid \calF[\bfG_s]}.\]

2. Let $\lambda$ be a character of $\calF[\bfG_o]$, such that $\lambda(\grun)\neq 0$.
Then $\lambda$ is invertible for the convolution product $\star$ associated to $\delta$: its inverse is denoted by $\mu$.
Denoting by $S$ the antipode of $(\calF[\bfG_o],m,\Delta)$, let us prove that $\lambda \circ S$ is invertible for $\star$
and that its inverse is $\mu\star (\epsilon_\delta\circ S)$. Firstly, $\lambda \circ S(\grun)=-\lambda(\grun)\neq 0$,
so $\lambda \circ S$ is invertible. Moreover, 
\begin{align*}
(\lambda \circ S)\star \mu \star (\epsilon_\delta\circ S)&=(\lambda \otimes \mu\circ \epsilon_\delta)
\circ (S\otimes \id \otimes S)\circ (\delta \otimes \id)\circ \delta\\
&=(\lambda \otimes \mu\circ \epsilon_\delta)
\circ (\id\otimes \id \otimes S)\circ (\delta \otimes \id)\circ \delta\circ S\\
&=(\lambda \otimes \mu\circ \epsilon_\delta)\circ (\delta \otimes \id)\circ (\id \otimes S)\circ \delta\circ S\\
&=(\lambda \star \mu \otimes \epsilon_\delta)\circ (\id \otimes S)\circ \delta \circ S\\
&=\epsilon_\delta\otimes \epsilon_\delta \circ (\id \otimes S)\circ \delta \circ S\\
&=\epsilon_\delta\circ S^2\\
&=\epsilon_\delta.
\end{align*}
We used for the second equality that $(S\otimes \id)\circ \delta=\delta \circ S$ (see \cite[Proposition 2.1]{Foissy40})
and, for the last equality, that $S^2=\id$, as $\calF[\bfG]$ is commutative.
So $\mu\star(\l\epsilon_\delta \circ S)=(\lambda\circ S)^{\star -1}$.\\

In the particular case were $\lambda=\epsilon_\delta$, then $\mu=\epsilon_\delta$ and we obtain that 
$(\epsilon_\delta\circ S)^{\star -1}=\epsilon_\delta\circ S$. \\

Let $G$ be an oriented graph. By definition of the weak Ehrhart polynomial and by the duality principle for Ehrhart polynomial
(Corollary \ref{cor4.7}),
\begin{align*}
\lambda_W(G)&=1\\
&=\Ehr_W(G)(1)\\
&=(-1)^{|V(G)|}\Ehr(G)(-1)\\
&=(-1)^{|V(G)|}S\circ \Ehr(G)(1)\\
&=(-1)^{|V(G)|}\Ehr(S(G))(1)\\
&=(-1)^{|V(G)|}\epsilon_\delta\circ S(G). 
\end{align*}

We now consider the three characters of $\calF[\bfG_o]$ defined on any oriented graph $G$ by
\begin{align*}
\lambda(G)&=\epsilon_\delta\circ S(G),&\mu(G)&=(-1)^{|V(G)|}\lambda(G),&\nu(G)&=(-1)^{\cc(G)}\lambda(G).
\end{align*}
We already proved that $\lambda \star \lambda=\epsilon_\delta$. For any graph $G$,
\begin{align*}
\mu \star \nu(G)&=\sum_{\sim\in \eq^c[G]} (-1)^{|V(G/\sim)|-\cc(G\mid \sim)}\lambda(G/\sim)\lambda(G\mid \sim)\\
&=\sum_{\sim\in \eq^c[G]} (-1)^{2\cl(\sim)}\lambda(G/\sim)\lambda(G\mid \sim)\\
&=\sum_{\sim\in \eq^c[G]} \lambda(G/\sim)\lambda(G\mid \sim)\\
&=\lambda \star\lambda(G)\\
&=\epsilon_\delta(G).
\end{align*}
Therefore, $\nu=\mu^{\star-1}$. As $\mu=\lambda_W$, we obtain that $\nu=\nu_W$ and, for any oriented graph $G$,
\[\nu_W(G)=(-1)^{-\cc(G)}\lambda(G)=(-1)^{|V(G)|-\cc(G)}\lambda_W(G)=(-1)^{|V(G)|-\cc(G)}.\qedhere\]
\end{proof}

No interpretation of $\nu_W(G)$ is known in general.  For example:

\begin{prop}
Let $G_n$ be the mixed graph of Example \ref{ex3.1}, with the convention $G_2=\grdeuxo$.
For any $n\geq 2$, $\nu_W(G_n)=(-1)^{n-1}(n-1)$. 
\end{prop}

\begin{proof}
For any $n\geq 2$,
\[\nu_W*\lambda_W(G_n)=\sum_{\sim\in \eq^c[G_n]} \nu_W(G/\sim) \lambda_W(G\mid \sim)=\epsilon_\delta(G_n)=0.\]
By definition of $\lambda_W$, for any $\sim\in \eq^c[G_n]$, 
$\lambda_W(G\mid \sim)=0$ if, and only if, $1\sim n$. 
Therefore, the contributing terms corresponds to the equivalences whose classes are intervals of $[n]$,
at the exception of the one with only one class. 
For such an equivalence $\sim$, the quotient $G/\sim$ is isomorphic to $G_{\cl(\sim)}$. We obtain that if $n\geq 3$,
\[\sum_{k=2}^n \sum_{\substack{i_1+\ldots+i_k=n,\\ i_1,\ldots,i_k\geq 1}} \nu_W(G_k)=0.\]
A direct computation shows that $\lambda_W(G_2)=-1$. Summing, we obtain in the ring of formal series $\mathbb{Q}[[X]]$ that
\begin{align*}
\sum_{k=2}^\infty \sum_{\substack{i_1+\ldots+i_k=n,\\ i_1,\ldots,i_k\geq 1}}\nu_W(G_k) X^{i_1+\ldots+i_n}
=\sum_{k=2}^\infty \nu_W(G_k) \left(\frac{X}{1-X}\right)^k=-X^2.
\end{align*}
Substituting $\dfrac{X}{1+X}$ to $X$, we obtain
\[\sum_{n=2}^\infty \nu_W(G_n)X^n=-\dfrac{X^2}{(1+X)^2}=\sum_{n=2}^\infty (-1)^{n+1}(n-1)X^n.\qedhere\]
\end{proof}

\section{Appendix: number of mixed graphs up to isomorphism}

\begin{prop}
Let $N\in \N$, nonzero. For any $n\in \N$, we denote respectively by $G_m(N,n)$, $G_o(N,n)$ and $G_s(N,n)$ the number of isomorphism classes of mixed, oriented, simple graphs
with $n$ vertices decorated by elements of $[N]$.  Then, for $t\in \{m,o,s\}$,
\[G_t(N,n)=\sum_{\substack{c_1,\ldots,c_n \geq 0,\\ 1c_1+\ldots+nc_n=n}}\frac{N^{c_1+\ldots+c_n}}{1^{c_1}\ldots n^{c_n}c_1!\ldots c_n!}
A_t^{\displaystyle \sum_{1\leq k<l\leq n}c_kc_l(k\wedge l)+\sum_{1\leq k\leq n} \frac{c_k(kc_k-1)}{2}}
\left(\frac{B_t}{\sqrt{A_t}}\right)^{\displaystyle \sum_{1\leq k\leq n/2}c_{2k}},\]
with
\[\begin{array}{|c||c|c|}
\hline t&A_t&B_t\\
\hline \hline m&5&3\\
\hline o&4&2\\
\hline s&2&2\\
\hline \end{array}\]
\end{prop}

\begin{proof}
We put
\begin{align*}
F_m&=\{(0,0),(0,\rightarrow),(\rightarrow,0),(\rightarrow,\rightarrow),(-,-)\},\\
F_o&=\{(0,0),(0,\rightarrow),(\rightarrow,0),(\rightarrow,\rightarrow)\},\\
F_s&=\{(0,0),(-,-)\}.
\end{align*}
These sets give the possibilities of edges or arcs between two vertices of a graph of the considered graphs. For any $t\in \{m,o,s\}$, $F_t$ is given an involution $\iota$, given by the usual flip. 
Note that $A_t$ is the cardinal of $F_t$ and that $B_t$ is the number of elements of $F_t$ invariant under $\iota$. 
Then, a graph $G$ of type $t$, whose set of vertices is $[n]$, together with a map from $V(G)$ to $[N]$, can be seen as a pair $(d_V,d_E)$, where $d_V:[n]\longrightarrow [N]$ is a map
and $d_E=\mathcal{P}_2([n])\longrightarrow F_t$ is a map which makes explicit the edge and arc situation between two vertices $i,j$, with $i<j$. 
The symmetric group $\sym_n$ acts on these pairs $(d_V,d_E)$, in the following sense: for any $\sigma \in \sym_n$, putting $\sigma\cdot(d_V,d_E)=(\sigma \cdot d_V,\sigma \cdot d_E)$,
\begin{itemize}
\item $\sigma \cdot d_V=d_V\circ \sigma^{-1}$.
\item For any $i,j\in [n]$, with $i<j$,
\[\sigma \cdot d_E(\{i,j\})=\begin{cases}
d_E(\{\sigma^{-1}(i),\sigma^{-1}(j)\}) \mbox{ if }\sigma^{-1}(i)<\sigma^{-1}(j),\\
\iota \circ d_E(\{\sigma^{-1}(i),\sigma^{-1}(j)\}) \mbox{ if }\sigma^{-1}(i)>\sigma^{-1}(j).
\end{cases}\]
\end{itemize}
The number of isomorphism classes $G_t(N,n)$ we are looking for is the number of orbits of this action. Using Burnside's formula,
\[G_t(N,n)=\frac{1}{n!} \fix(\sigma),\]
where $\fix(\sigma)$ is the number of pairs $(d_V,d_E)$ invariant under the action of $\sigma$. \\

Let $\sigma \in \sym_n$. Then $\sigma \cdot (d_V,d_E)=(d_V,d_E)$ if, and only if:
\begin{itemize}
\item For any $i,j$ in the same $\sigma$-orbit, $d_V(i)=d_V(j)$. The number of $\sigma$-orbits is denoted by $n(\sigma)$. 
\item For any orbit of the action of $\sigma$ on pairs $\{i,j\}$, with $i<j$, such that there exists $k\in \N$ with $\sigma^k(i)=j$ and $\sigma^k(j)=i$, 
$d_E$ is constant on this orbit and its values is invariant under $\iota$.
Such an orbit will be called an orbit of type $B$. The number of orbits of type $B$ is denoted by $n_B(\sigma)$.
\item For any orbit of the action of $\sigma$ on pairs $\{i,j\}$, with $i<j$, such that for any $k\in \N$, $(\sigma^k(i),\sigma^k(j))\neq(j,i)$, $d_E$ is constant on this orbit.
Such an orbit will be called an orbit of type $A$. The number of orbits of type $A$ is denoted by $n_A(\sigma)$.
\end{itemize}
Consequently, 
\[\fix(\sigma)=N^{n(\sigma)}A^{n_A(\sigma)}B^{n_B(\sigma)}.\]

Let us decompose $\sigma$ into cycles with disjoint supports. For any $i\in [N]$, the number of cycles of length $i$ in this decomposition is denoted by $c_i$. Then $1c_1+\ldots+nc_n=n$.
Then $n(\sigma)=c_1+\ldots+c_n$. 

Let $\{i,j\}$ be a pair such that $i,j$ belong to two different cycles of $\sigma$. As $\sigma^k(i)$ belongs to the cycle containing $i$ for any $k\in \N$, $\sigma^k(i)\neq j$. So the orbit of $\{i,j\}$ if of type $A$.
Moreover, denoting by $k$ and $l$ the lengths of the cycle containing $i$ and $j$, this orbit is of cardinality $k\vee l$. Hence, these two cycles give $\dfrac{kl}{k\vee l}=k\wedge k$ orbits of type $A$.
Summing over all possible choices of disjoint cycles of $\sigma$, we obtain
\[\sum_{1\leq k<l \leq n}c_kc_k (k\wedge l)+\sum_{1\leq k\leq n} \dfrac{c_k(c_k-1)}{2}k\]
orbits of type $A$.
 
Let $\{i,j\}$ be a pair such that $i,j$ belong to the same cycle of $\sigma$. The length of this cycle is denoted by $l$: for a given cycle, this gives $\dfrac{l(l-1)}{2}$ pairs. 
Moreover, for any $k\in \N$,
\begin{align*}
\{\sigma^k(i),\sigma^k(j)\}=\{i,j\}&\Longleftrightarrow \begin{cases}
\sigma^k(i)=i\\
\sigma^k(j)=j
\end{cases}\mbox{ or }
\begin{cases}
\sigma^k(i)=j\\
\sigma^k(j)=i.
\end{cases}
\end{align*}
Moreover, $\sigma^k(i)=i$ if and only if, $k$ is a multiple of $l$. If $\sigma^k(i)=j$ and $\sigma^k(j)=i$, then $\sigma^{2k}(i)=i$, so $2k$ is a multiple of $l$. 
We obtain two cases:
\begin{itemize}
\item If $l$ is odd, then If $\sigma^k(i)=j$ and $\sigma^k(j)=i$, $l\mid 2k$ and, as $l$ is odd, $l\mid k$. In other words,
\[\{\sigma^k(i),\sigma^k(j)\}=\{i,j\}\Longleftrightarrow l\mid k.\]
As a consequence, all the orbits contained in this cycle are of type $A$, and of cardinality $l$. This gives $\dfrac{l-1}{2}$ orbits of type $A$.
\item If $l$ is odd, we obtain an orbit of type $B$: if the cycle is $(i_1,\ldots,i_{2p})$, this orbit is \[\{\{i_1,i_{p+1}\},\ldots \{i_p,i_{2p}\}\}.\]
It contains $p=\dfrac{l}{2}$ elements. All the other orbits included in this cycle have cardinality $l$, so we obtain $\dfrac{1}{l}\left(\dfrac{l(l-1)}{2}-\dfrac{l}{2}\right)=\dfrac{l-2}{2}$ orbits of type $A$.
\end{itemize}
Summing over all possible choices of cycles, we obtain $c_2+c_4+\ldots+$ orbits of type $B$ and 
\begin{align*}
\sum_{\substack{1\leq k\leq n,\\ \mbox{\scriptsize $k$ odd}}} c_k\frac{k-1}{2}
+\sum_{\substack{1\leq k\leq n,\\ \mbox{\scriptsize $k$ even}}} c_k\frac{k-2}{2}.
\end{align*}

Finally,
\begin{align*}
n_B(\sigma)&=\sum_{1\leq k\leq n/2} c_{2k},\\
n_A(\sigma)&=\sum_{1\leq k<l \leq n}c_kc_k (k\wedge l)+\sum_{1\leq k\leq n} \dfrac{c_k(c_k-1)}{2}k
+\sum_{\substack{1\leq k\leq n,\\ \mbox{\scriptsize $k$ odd}}} c_k\frac{k-1}{2}
+\sum_{\substack{1\leq k\leq n,\\ \mbox{\scriptsize $k$ even}}} c_k\frac{k-2}{2}\\
&=\sum_{1\leq k<l \leq n}c_kc_k (k\wedge l)+\sum_{1\leq k\leq n}\frac{(kc_k-1)c_k}{2}-\frac{1}{2}\sum_{1\leq k\leq n/2} c_{2k},
\end{align*}
which finally gives the announced formula, as there are $1^{c_1} \ldots n^{c_n}c_1!\ldots c_n!$ permutations whose decomposition in cycles with disjoint supports is given by $c_i$ cycles of length $i$ for any $i\in [n]$.
\end{proof}

\begin{example} For a given $n$, $G_t(N,n)$ is a polynomial in $N$, $A_t$ and $B_t$.
\begin{align*}
G_t(N,1)&=N,\\
G_t(N,2)&=\frac{(A_tN+B_t)N}{2},\\
G_t(N,3)&=\frac{(A_t^2N^2+3B_tN+2)A_tN}{6},\\
G_t(N,4)&=\frac{(A_t^5N^3+6A_t^2B_tN^2+8A_tN+6B_t)A_tN}{24}.
\end{align*}
Tables \ref{table4}-\ref{table6} give first values of $G_t(N,n)$. 
\end{example}

\begin{table}[h]\label{table4}
\[\begin{array}{|c||c|c|c|c|c|c|}
\hline n\setminus N&1&2&3&4&5&6\\
\hline \hline 1&1&2&3&4&5&6\\
\hline 2&4&13&27&46&70&99\\
\hline 3&30&200&635&1460&2800&4780\\
\hline 4&785&11320&55605&173265&419550&865335\\
\hline 5&86130&2673260&20113890&84385520&256856275&638050530\\
\hline 6&43053850&2733053500&31051529575&174249075200&664212533500&1982349763225\\
\hline \end{array}\]
The second row is entry 
\href{https://oeis.org/A147875}{A147875} of the OEIS \cite{Sloane}.
\caption{Number of mixed graphs}
\end{table}

\begin{table}[h]
\[\begin{array}{|c||c|c|c|c|c|c|}
\hline n\setminus N&1&2&3&4&5&6\\
\hline \hline 1&1&2&3&4&5&6\\
\hline 2&3&10&21&36&55&78\\
\hline 3&16&104&328&752&1440&2456\\
\hline 4&218&3044&14814&45960&111010&228588\\
\hline 5&9608&291968&2183400&9133760&27755016&68869824\\
\hline 6&1540944&96928992&1098209328&6154473664&23441457680&69924880288\\
\hline\end{array}\]
The second row is entry 
 \href{https://oeis.org/A014105}{A014105}
 of the OEIS, the first and second columns are entries 
 \href{https://oeis.org/A000273}{A000273} and  \href{https://oeis.org/A000595}{A000595}.
\caption{Number of oriented graphs}
\end{table}

\begin{table}[h]\label{table6}
\[\begin{array}{|c||c|c|c|c|c|c|}
\hline n\setminus N&1&2&3&4&5&6\\
\hline \hline 1&1&2&3&4&5&6\\
\hline 2&2&6&12&20&30&42\\
\hline 3&4&20&56&120&220&364\\
\hline 4&11&90&357&996&2255&4446\\
\hline 5&34&544&3258&12208&34754&82608\\
\hline 6&156&5096&47324&241520&871580&2510424\\
\hline \end{array}\]
The second, third and fourth rows are respectively entries \href{https://oeis.org/A002378}{A002378}, 
\href{https://oeis.org/A002492}{A002492}
and \href{https://oeis.org/A199394}{A199394} of the OEIS. 
The first and second columns are entries 
\href{https://oeis.org/A000088}{A000088} and \href{https://oeis.org/A000666}{A000666}.
\caption{Number of simple graphs}
\end{table}

\bibliographystyle{amsplain}
\bibliography{biblio}

\end{document}